\documentclass[11pt]{amsart}

\usepackage{subcaption}
\usepackage{subfiles}
\usepackage{tabu}
\usepackage[margin=1in]{geometry}
\usepackage[dvipsnames]{xcolor}   		             		
\usepackage{graphicx}		
\usepackage{amssymb}
\usepackage{mathrsfs}
\usepackage{colonequals}
\usepackage{amsthm}
\usepackage{amsmath}
\usepackage{stmaryrd}
\usepackage{tikz}
\usepackage{tikz-cd}
\usepackage{accents}
\usepackage{upgreek}
\usepackage{enumerate}
\usepackage{bm}
\usepackage{mathtools}
\usepackage[all]{xy}
\usepackage{caption}
\usepackage{subcaption}

\usepackage{url}
\usepackage{float}
\usepackage{todonotes}
\usepackage{bbm}

\usepackage{longtable}

\usepackage[full]{textcomp}
\usepackage[cal=cm]{mathalfa}
\usepackage{xparse}
\usepackage{comment}
\usepackage{cite}

\usepackage{enumitem}

\tikzset{
  commutative diagrams/.cd, 
  arrow style=tikz, 
  diagrams={>=stealth}
}

\theoremstyle{theorem}
\newtheorem{innercustomthm}{Theorem}
\newenvironment{customthm}[1]
  {\renewcommand\theinnercustomthm{#1}\innercustomthm}
  {\endinnercustomthm}
  
  \theoremstyle{definition}
\newtheorem{innercustomdef}{Definition}
\newenvironment{customdef}[1]
  {\renewcommand\theinnercustomdef{#1}\innercustomdef}
  {\endinnercustomdef}
  
  \theoremstyle{definition}
\newtheorem{innercustomquestion}{Question}
\newenvironment{customquestion}[1]
  {\renewcommand\theinnercustomquestion{#1}\innercustomquestion}
  {\endinnercustomquestion}
  
  \theoremstyle{definition}
\newtheorem{innercustomconj}{Conjecture}
\newenvironment{customconj}[1]
  {\renewcommand\theinnercustomconj{#1}\innercustomconj}
  {\endinnercustomconj}
  
\makeatletter
\def\@tocline#1#2#3#4#5#6#7{\relax
  \ifnum #1>\c@tocdepth 
  \else
    \par \addpenalty\@secpenalty\addvspace{#2}%
    \begingroup \hyphenpenalty\@M
    \@ifempty{#4}{%
      \@tempdima\csname r@tocindent\number#1\endcsname\relax
    }{%
      \@tempdima#4\relax
    }%
    \parindent\z@ \leftskip#3\relax \advance\leftskip\@tempdima\relax
    \rightskip\@pnumwidth plus4em \parfillskip-\@pnumwidth
    #5\leavevmode\hskip-\@tempdima
      \ifcase #1
       \or\or \hskip 1em \or \hskip 2em \else \hskip 3em \fi%
      #6\nobreak\relax
    \dotfill\hbox to\@pnumwidth{\@tocpagenum{#7}}\par
    \nobreak
    \endgroup
  \fi}
\makeatother

\usetikzlibrary{calc}
\usetikzlibrary{fadings}
\usetikzlibrary{decorations.pathmorphing}
\usetikzlibrary{decorations.pathreplacing}

\newcounter{marginnote}
\DeclareMathAlphabet{\mathpzc}{OT1}{pzc}{m}{it}

\usepackage[backref=page]{hyperref}
\hypersetup{
  colorlinks   = true,          
  urlcolor     = olive,          
  linkcolor    = olive,          
  citecolor   = olive             
}

\theoremstyle{theorem}
\newtheorem{theorem}{Theorem}[section]

\newtheorem{lemma}[theorem]{Lemma}
\newtheorem{proposition}[theorem]{Proposition}

\theoremstyle{definition}
\newtheorem*{conjecture*}{Conjecture}
\newtheorem{remark}[theorem]{Remark}

\newtheorem*{runningexample*}{Running example}

\newtheorem*{aside*}{Aside}

\newtheorem{definition}[theorem]{Definition}
\newtheorem{definition-construction}[theorem]{Definition-Construction}
\newtheorem{example}[theorem]{Example}

\newtheorem{proposition-definition}[theorem]{Proposition-Definition}
\newtheorem{question}[theorem]{Question}

\DeclareMathOperator{\Id}{Id}

\newcommand{\Ycal}{\mathcal{Y}}

\newcommand{\Gm}{\mathbb{G}_{\mathrm{m}}}

\newcommand{\bcd}{\begin{center}\begin{tikzcd}}
\newcommand{\ecd}{\end{tikzcd}\end{center}}

\newcommand{\PP}{\mathbb{P}}

\newcommand{\Z}{\mathbb{Z}}

\newcommand{\R}{\mathbb{R}}

\newcommand{\Ccal}{\mathcal{C}}

\newcommand{\Spec}{\operatorname{Spec}}
\newcommand{\acts}{\curvearrowright}

\NewDocumentCommand{\compatibilitydatum}{m m m m m m O{} O{} O{}}{
\begin{equation*} \begin{tikzcd}[ampersand replacement=\&]
  \: \arrow{r} \& {#1} \arrow{r} \arrow{d}{#7} \& {#2} \arrow{r} \arrow{d}{#8} \& {#3} \arrow{r}{[1]} \arrow{d}{#9} \& \: \\
  \: \arrow{r} \& {#4} \arrow{r} \& {#5} \arrow{r} \& {#6} \arrow{r} \& \:
\end{tikzcd} \end{equation*}}

\NewDocumentCommand{\commutingsquare}{m m m m o O{} O{} O{} O{}}{
\begin{equation}\begin{tikzcd}[ampersand replacement=\&] \label{#5}
  #1 \arrow{r}{#6} \arrow{d}{#7} \& #2 \arrow{d}{#8} \\
  #3 \arrow{r}{#9} \& #4
\end{tikzcd}\IfValueTF{#5}{\label{#5}}{} \end{equation}}

\NewDocumentCommand{\cartesiansquare}{m m m m O{} O{} O{} O{}}{
\begin{equation*}\begin{tikzcd}[ampersand replacement=\&]
  #1 \arrow{r}{#5} \arrow{d}{#6} \arrow[dr, phantom, "\square"] \& #2 \arrow{d}{#7} \\
  #3 \arrow{r}{#8} \& #4
\end{tikzcd} \end{equation*}}

\NewDocumentCommand{\cartesiansquarelabel}{m m m m m O{} O{} O{} O{}}{
\begin{tikzcd}[ampersand replacement=\&]
  #1 \arrow{r}{#6} \arrow{d}{#7} \arrow[dr, phantom, "\square"] \& #2 \arrow{d}{#8} \\
  #3 \arrow{r}{#9} \& #4
\end{tikzcd}\IfValueTF{#5}{\label{#5}}{}
}

\NewDocumentCommand{\triangleofspaces}{m m m O{} O{} O{}}{
\begin{tikzcd} [ampersand replacement=\&]
#1 \arrow{r}{#4} \arrow[bend right]{rr}{#5} \& #2 \arrow{r}{#6} \& #3
\end{tikzcd}}

\begin{document}
 
\title[Toric configuration spaces]{Toric configuration spaces:\\ the bipermutahedron and beyond}
\author{Navid Nabijou}

\begin{abstract}
We establish faithful tropicalisation for point configurations on algebraic tori. Building on ideas from enumerative geometry, we introduce tropical scaffolds and use them to construct a system of modular fan structures on the tropical configuration spaces. The corresponding toric varieties provide modular compactifications of the algebraic configuration spaces, with boundary parametrising transverse configurations on tropical expansions. The rubber torus, used to identify equivalent configurations, plays a key role. As an application, we obtain a modular interpretation for the bipermutahedral variety.
\end{abstract}

\maketitle
\vspace{-25pt}
\tableofcontents

\vspace{-30pt}
\section*{Introduction} 
\noindent Fix a lattice $N\cong\Z^d$ and consider the associated algebraic torus and real vector space
\begin{align*} T \colonequals N \otimes \Gm, \qquad V \colonequals N \otimes \R.
\end{align*}
This paper explores the relationship between two moduli problems:
\begin{align*}
\text{Algebraic: \quad Labelled points $x_0,\ldots,x_n \in T$ up to simultaneous translation.}\\
\text{Tropical: \quad Labelled points $p_0,\ldots,p_n \in V$ up to simultaneous translation.}
\end{align*}
Points may coincide. Trivially, the corresponding moduli spaces are products modulo diagonals:
\begin{align*}
T[n]  \colonequals T^{n+1}/T, \qquad V[n] \colonequals V^{n+1}/V.
\end{align*}
Consider $N[n] \colonequals N^{n+1}/N$. We have $T[n] = N[n]\otimes \Gm$ and $V[n] = N[n] \otimes \R$. Consequently, there is a bijection between toric compactifications of $T[n]$ and complete fan structures on $V[n]$. In this paper we refine this into a correspondence:
\begin{equation} \label{eqn: correspondence} \left\{ \text{\parbox{0.185\textwidth}{\center{modular \\ compactifications \\ of $T[n]$}}} \right\} \longleftrightarrow \left\{ \text{\parbox{0.185\textwidth}{\center{modular \\ fan structures \\ on $V[n]$}}} \right\}.
\end{equation}
A key step is identifying the appropriate class of modular fan structures. The case $d=1$ recovers the identification between the Losev--Manin moduli space and the permutahedral variety. The cases $d \geq 2$ are new. The case $d=2$ in particular furnishes modular interpretations for the toric varieties associated to the bipermuthedron and the square of the permutahedron.

The above correspondence provides another example of faithful tropicalisation, in which tropicalising a moduli space of algebraic objects produces a moduli space of associated tropical objects \cite{CaporasoGonality,ACP,UlirschWeightedStable,AscherMolcho,CMRAdmissible,CHMRWeighted,GrossCorrespondence,RanganathanSkeletons1,CCUW, LenUlirsch,OdakaOshima,MW,TropUnivJac,KHQuot,BBCMMW,MR20,BCK,DivisorsCurves}.

For simplicity we use the translation action to fix $x_0$ as the identity element $1 \in T$. We refer to this as the \textbf{anchor point}. This produces an identification
\[ T[n] \cong T^n\]
consisting of the moduli for the remaining points $x_1,\ldots,x_n$. We retain the anchor point $x_0=1 \in T$  to align with the existing literature.

\subsection{Expansions and scaffolds} Fix a valuation ring $R$ with fraction field $K$ and value group $\R$. Consider a family of point configurations in $T$
\[ \Spec K \to T[n].\]
 Coordinatewise valuation produces a point configuration in $V$
\[ (x_1,\ldots,x_n) \in T[n](K) \quad \rightsquigarrow \quad (p_1,\ldots,p_n) \in V[n].\]
This intertwines algebraic and tropical moduli. The valuation $p_i \in V$ records the asymptotics of the point $x_i \in T(K)$. Since $x_0=1 \in T$ we have $p_0=0 \in V$.

Insights from logarithmic enumerative geometry (see e.g. \cite[Proposition~6.3]{NishinouSiebert} or \cite[Lemma~7.2.4]{MR20}) suggest that as the points $x_1,\ldots,x_n$ approach infinity, the ambient variety $T$ should degenerate to a \textbf{tropical expansion}: a union of toric varieties meeting along toric strata. The points $x_1,\ldots,x_n$ will then limit to interior points on the irreducible components of this expansion. A configuration on an expansion consisting entirely of interior points is referred to as \textbf{transverse}.

A tropical expansion is encoded by a polyhedral decomposition of $V$, whose vertices index the irreducible components. To ensure that the point $x_i \in T(K)$ limits to the interior of an irreducible component, its valuation $p_i \in V$ must lie on a vertex of this polyhedral decomposition. The situation is summarised in Figure~\ref{fig: limit}.

Producing a modular compactification of $T[n]$ thus requires the following input data  (Section~\ref{sec: scaffold}).

\begin{customdef}{X}[$\approx$ {Definition~\ref{def: scaffold}}] A \textbf{tropical scaffold} is the data of, for every point $(p_1,\ldots,p_n) \in V[n]$, a polyhedral decomposition of $V$ containing the points $p_0,p_1,\ldots,p_n \in V$ as vertices.	
\end{customdef}
This data is required to vary linearly with the point $(p_1,\ldots,p_n)$ and is thus encoded as a complete fan $\Lambda$ on $V[n] \times V$ whose fibre over the point $(p_1,\ldots,p_n) \in V[n]$ gives the corresponding polyhedral decomposition of $V$.

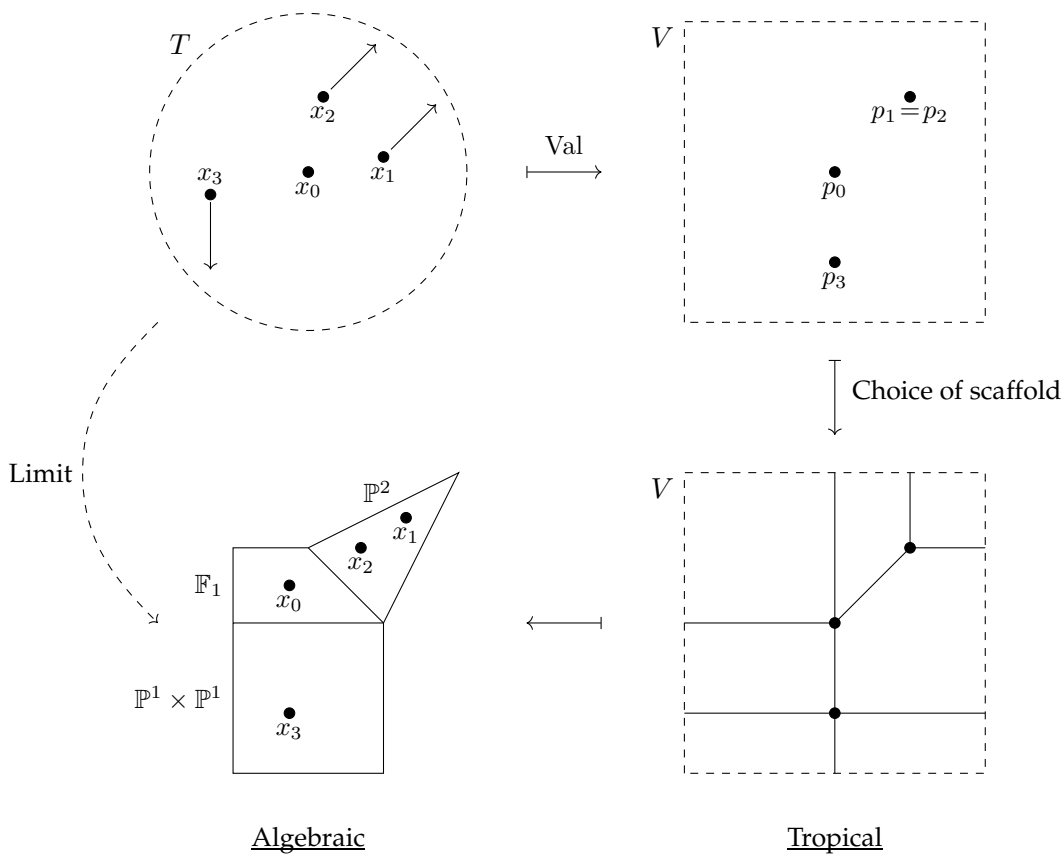
\begin{figure}
\begin{tikzpicture}


\draw(0,0) [dashed] circle[radius=60pt];
\draw(-1.7,1.7) node{$T$};

\draw (0,0) [fill=black] circle[radius=2pt];
\draw(0,0) node[below]{\small$x_0$};

\draw (1,0.2) [fill=black] circle[radius=2pt];
\draw (1,0.2) node[below]{\small$x_1$};
\draw [->] (1.1,0.3) -- (1.7,0.9);

\draw (0.2,1) [fill=black] circle[radius=2pt];
\draw (0.2,1) node[below]{\small$x_2$};
\draw [->] (0.3,1.1) -- (0.9, 1.7);

\draw (-1.3,-0.3) [fill=black] circle[radius=2pt];
\draw (-1.3,-0.3) node[above]{\small$x_3$};
\draw [->] (-1.3,-0.4) -- (-1.3,-1.3);

\draw [->,dashed] (-2,-2) to[in=90,out=225] (-3,-4) to[out=270,in=135] (-2,-6);
\draw (-3,-4) node[left]{\small{Limit}};

\draw [|->] (2.9,0) -- (3.9,0);
\draw (3.4,0.1) node[above]{\small{$\operatorname{Val}$}};


\draw [dashed] (5,-2) -- (5,2) -- (9,2) -- (9,-2) -- (5,-2);
\draw (5,1.8) node[left]{$V$};

\draw[fill=black] (7,0) circle[radius=2pt];
\draw (7,0) node[below]{\small$p_0$};

\draw[fill=black] (8,1) circle[radius=2pt];
\draw (8,1) node[below]{\small$p_1\!=\!p_2$};

\draw[fill=black] (7,-1.2) circle[radius=2pt];
\draw (7,-1.2) node[below]{\small$p_3$};

\draw [|->] (7,-2.5) -- (7,-3.5);
\draw (7.1,-2.9) node[right,align=center]{\small{Choice of scaffold}};


\draw [dashed] (5,-8) -- (5,-4) -- (9,-4) -- (9,-8) -- (5,-8);
\draw (5,-4.2) node[left]{$V$};

\draw[fill=black] (7,-6) circle[radius=2pt];

\draw[fill=black] (8,-5) circle[radius=2pt];

\draw[fill=black] (7,-7.2) circle[radius=2pt];

\draw (8,-5) -- (8,-4);
\draw (8,-5) -- (9,-5);
\draw (7,-6) -- (7,-4);
\draw (7,-6) -- (8,-5);
\draw (7,-6) -- (5,-6);
\draw (7,-6) -- (7,-7.2);
\draw (7,-7.2) -- (5,-7.2);
\draw (7,-7.2) -- (9,-7.2);
\draw (7,-7.2) -- (7,-8);

\draw [|->] (3.9,-6) -- (2.9,-6);

\draw (-1,-6) -- (1,-6) -- (1,-8) -- (-1,-8) -- (-1,-6);
\draw (-1,-7) node[left]{\small$\PP^1 \times \PP^1$};

\draw (-1,-6) -- (-1,-5) -- (0,-5) -- (1,-6);
\draw (-1,-5.5) node[left]{\small$\mathbb{F}_1$};

\draw (1,-6) -- (2,-4) -- (0,-5);
\draw (0.6,-4.25) node[right]{\small$\PP^2$};

\draw[fill=black] (-0.25,-5.5) circle[radius=2pt];
\draw (-0.25,-5.5) node[below]{\small$x_0$};

\draw[fill=black] (1.3,-4.6) circle[radius=2pt];
\draw (1.3,-4.6) node[below]{\small$x_1$};

\draw [fill=black] (0.7,-5) circle[radius=2pt];
\draw (0.7,-5) node[below]{\small$x_2$};

\draw[fill=black] (-0.25,-7.2) circle[radius=2pt];
\draw (-0.25,-7.2) node[below]{\small$x_3$};

\draw (0,-9) node{\small{\underline{\smash{Algebraic}}}};

\draw (7,-9) node{\small{\underline{\smash{Tropical}}}};

\end{tikzpicture}
\caption{A configuration on $T$ limiting to a transverse configuration on an expansion.}
\label{fig: limit}
\end{figure}

\subsection{Moduli} In Section~\ref{sec: moduli} we fix a tropical scaffold and construct the associated configuration space. Our results are summarised as follows.

\begin{customthm}{Y} To each tropical scaffold $\Lambda$ there is an associated configuration space $P_{\Lambda}$. This is a toric Deligne--Mumford stack compactifying $T[n]$, and its stacky fan is constructed explicitly from $\Lambda$ via universal weak semistable reduction (Definition~\ref{def: main construction}). It supports a universal tropical expansion and transverse point configuration (Section~\ref{sec: universal family}):
\[
\begin{tikzcd}
\Ycal_\Lambda \ar[d,"\uppi" left] \\
P_\Lambda. \ar[u,bend right, "{x_0,x_1,\ldots,x_n}" right]
\end{tikzcd}
\]
Each locally-closed stratum of $P_\Lambda$ parametrises transverse point configurations on the associated tropical expansion, with two configurations identified if they differ by the action of the rubber torus (Theorem~\ref{thm: strata}).
\end{customthm}
The rubber torus originates in logarithmic enumerative geometry \cite[Theorem~1.8]{CarocciNabijouRubber}. It is a crucial ingredient, necessary in order to obtain a separated moduli space. This perhaps explains why the above construction was not discovered earlier.

We build the fan of the configuration space $P_\Lambda$ explicitly from the scaffold $\Lambda$ (Section~\ref{sec: main construction}). Conceptually, it is obtained by stratifying $V[n]$ into regions on which the underlying polyhedral complex of the polyhedral decomposition of $V$ is constant. Practically, the construction $\Lambda \rightsquigarrow P_\Lambda$ can be viewed both as an instance of universal weak semistable reduction and as an example of a Chow quotient (Section~\ref{sec: semistable reduction}). This identifies the correct class of modular fan structures for \eqref{eqn: correspondence}.

While in general $P_\Lambda$ is a toric Deligne--Mumford stack, in most cases of interest it is a vanilla toric variety (but see Example~\ref{ex: stacky example} and Conjecture~\ref{conj: minimal gives a variety}).


\subsection{Bipermutahedral variety} \label{sec: examples introduction} This project began as an attempt to find a modular interpretation for the bipermutahedral variety. The bipermutahedral fan parametrises labelled points in $\R^2$ up to translation, stratified according to bisequence \cite{ArdilaDenhamHuh,ArdilaBipermutahedron}. The notion of bisequence suggests a particular choice of tropical scaffold. We identify the associated configuration space with the bipermutahedral variety (Figure~\ref{fig: biperm scaffold introduction} and Section~\ref{sec: bipermutahedron}). A coarsening of this scaffold also produces the square of the permutahedral variety (Figure~\ref{fig: 2 perms scaffold introduction} and Section~\ref{sec: square of permutahedron}).

This produces modular interpretations for these spaces, resembling the identification of the permutahedral variety with the Losev--Manin moduli space.
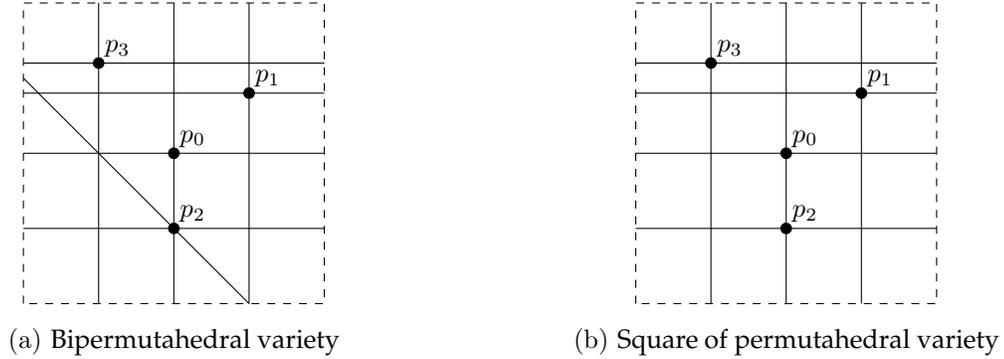
\begin{figure}
\centering
\begin{subfigure}[b]{0.4\textwidth}
\[
\begin{tikzpicture}


\draw [dashed] (5,-2) -- (5,2) -- (9,2) -- (9,-2) -- (5,-2);

\draw (5,0) -- (9,0);
\draw (5,0.8) -- (9,0.8);
\draw (8,2) -- (8,-2);
\draw (5,1.2) -- (9,1.2);
\draw (6,2) -- (6,-2);
\draw (7,2) -- (7,-2);
\draw (5,-1) -- (9,-1);

\draw (5,1) -- (8,-2);

\draw[fill=black] (7,0) circle[radius=2pt];
\draw (7.25,-0.05) node[above]{\small$p_0$};

\draw[fill=black] (8,0.8) circle[radius=2pt];
\draw (8.25,0.75) node[above]{\small$p_1$};

\draw[fill=black] (7,-1) circle[radius=2pt];
\draw (7.25,-1.05) node[above]{\small$p_2$};

\draw[fill=black] (6,1.2) circle[radius=2pt];
\draw (6.25,1.15) node[above]{\small$p_3$};

\end{tikzpicture}
\]
\caption{Bipermutahedral variety}
\label{fig: biperm scaffold introduction}
\end{subfigure}\qquad\qquad
\begin{subfigure}[b]{0.4\textwidth}
\[
\begin{tikzpicture}


\draw [dashed] (5,-2) -- (5,2) -- (9,2) -- (9,-2) -- (5,-2);

\draw (5,0) -- (9,0);
\draw (5,0.8) -- (9,0.8);
\draw (8,2) -- (8,-2);
\draw (5,1.2) -- (9,1.2);
\draw (6,2) -- (6,-2);
\draw (7,2) -- (7,-2);
\draw (5,-1) -- (9,-1);

\draw[fill=black] (7,0) circle[radius=2pt];
\draw (7.25,-0.05) node[above]{\small$p_0$};

\draw[fill=black] (8,0.8) circle[radius=2pt];
\draw (8.25,0.75) node[above]{\small$p_1$};

\draw[fill=black] (7,-1) circle[radius=2pt];
\draw (7.25,-1.05) node[above]{\small$p_2$};

\draw[fill=black] (6,1.2) circle[radius=2pt];
\draw (6.25,1.15) node[above]{\small$p_3$};

\end{tikzpicture}
\]
\caption{Square of permutahedral variety}
\label{fig: 2 perms scaffold introduction}
\end{subfigure}
\caption{Scaffolds producing the bipermutahedral variety and the square of the permutahedral variety. The former is obtained from the latter by slicing with the supporting antidiagonal.}
\label{fig: dim 2 scaffolds introduction}
\end{figure}


\subsection{Prospects} The freedom to choose a tropical scaffold affords great flexibility. This is by design: the resulting class of configuration spaces is broad enough to include both the bipermutahedral variety and the square of the permutahedral variety.

Despite this flexibility, there is an open structural question. Tropical scaffolds form an inverse system: a refinement of a scaffold is again a scaffold. For $d=1$ and $n$ fixed, this inverse system has a unique minimal element, whose associated configuration space is the permutahedral variety (Section~\ref{sec: permutahedron}). This singles out permutahedral varieties amongst all toric Deligne--Mumford stacks. For $d \geq 2$ there are multiple minimal scaffolds.

\begin{customquestion}{A} \label{question: minimal} For $d \geq 2$, characterise the toric stacks arising as configuration spaces associated to minimal tropical scaffolds.	
\end{customquestion}
The scaffold giving rise to the square of the permutahedral variety is minimal, while the scaffold giving rise to the bipermutahedral variety is not. We suspect that the latter cannot be obtained from \emph{any} minimal scaffold, though at present we have no avenues for proving this. We also posit:
\begin{customconj}{B} \label{conj: minimal gives a variety} If a tropical scaffold is minimal, the associated toric stack is in fact a toric variety. \end{customconj}

Besides the bipermutahedral variety, there are several other toric varieties generalising the Losev--Manin moduli space. Examples include the harmonic variety (see \cite{ArdilaEscobar} and Section~\ref{sec: harmonic}) and the toric stacks associated to root systems of type B, C, D \cite{BatyrevBlume} (we thank the anonymous referee for drawing our attention to these latter spaces).

\begin{customquestion}{C} Which of the above examples arise as a configuration space $P_\Lambda$ associated to some tropical scaffold $\Lambda$? \end{customquestion}

\subsection{Proximate moduli}

\subsubsection{Toroidal embeddings} This paper may be adapted to study point configurations on toroidal embeddings $(X|D)$. Here $T$ is replaced by the interior $U \colonequals X \setminus D$ and $V$ is replaced by the tropicalisation $\Sigma \colonequals \Sigma(X|D)$. The universal tropical family is the projection
\[
\begin{tikzcd}
	\Sigma^n \times \Sigma \ar[d,"\uppi" left] \\
	\Sigma^n \ar[u,bend right, "{p_1,\ldots,p_n}" right]
\end{tikzcd}
\]
and there is no anchor point $p_0$ or translation action (compare with \eqref{eqn: universal family on vector spaces}). Tropical scaffolds $\Lambda$ are refinements of $\Sigma^n \times \Sigma$ and universal weak semistable reduction \cite[Section~3.2]{MolchoSS} produces $\Pi_\Lambda$.

This is not quite a generalisation of our setup, as it requires an initial choice of compactification $X$ of $U$. Consequently the configuration space is always a modification of $X^n$. We focus on tori in this paper in order to emphasise the pleasant combinatorics.

\subsubsection{Very affine varieties} This paper may also be adapted to study point configurations on a closed subvariety $U$ of an algebraic torus. Here $T$ is replaced by $U$ and $V$ is replaced by $\operatorname{trop}U \subseteq V$. Each tropical scaffold should produce a stacky fan structure on $(\operatorname{trop}U)^n = \operatorname{trop}(U^n)$. This will require a mild generalisation of the stacky Chow quotient of \cite[Section~3.2]{AscherMolcho}.

The set $\operatorname{trop} U$ admits a fan structure \cite[Theorem~2.2.5]{EKL}. If $U$ is sch\"{o}n then a choice of fan structure produces a toroidal embedding $(\overline{U}\ | \ \overline{U} \setminus U)$ which reduces this case to the previous one \cite[Theorem~1.4]{Tevelev}. However there is no minimal choice of fan structure on $\operatorname{trop} U$.

\subsubsection{Di Rocco--Schaffler} While preparing this paper we discovered \cite{DiRoccoSchaffler}. Building on \cite{GerritzenPiwek,KTChow,HKTHyperplane,AscherMolcho,SchafflerTevelev} the authors study toric configuration spaces of a similar flavour to ours: the configuration space is also constructed as a Chow quotient, and the boundary also parametrises configurations on tropical expansions.

The difference lies in the input data. Instead of a tropical scaffold, the authors begin with a complete fan $\Sigma$ on $N$. In our language, $\Sigma$ gives rise to a canonical tropical scaffold: for each $(p_1,\ldots,p_n) \in V[n]$ we overlay the translates of $\Sigma$ at the points $p_0,p_1,\ldots,p_n$ and take the minimal common refinement (see \cite[Definition~3.4]{DiRoccoSchaffler}). Under the identification $V[n] \times V \cong V^n \times V = V^{n+1}$ this gives $\Lambda = \Sigma^{n+1}$ and the associated configuration space is the Chow quotient by the diagonal. The configuration spaces studied in \cite{DiRoccoSchaffler} thus constitute an important special case of our construction, and the authors establish many interesting results in this setting.

Applying their construction with $\Sigma=\Sigma(\PP^1 \times \PP^1)$ gives the scaffold producing the square of the permutahedral variety (Section~\ref{sec: square of permutahedron}). On the other hand no choice of $\Sigma$ gives the scaffold producing the bipermutahedral variety (Section~\ref{sec: bipermutahedron}).

To our understanding, the rubber torus does not appear in \cite{DiRoccoSchaffler}.

\subsubsection{Distinct points} Consider configurations of distinct labelled points on a smooth projective variety $X$. The resulting moduli space is $X^n$ minus all diagonals. Fulton--MacPherson construct a beautiful modular compactification of this space \cite{FM}. As points collide, $X$ is replaced by the degeneration to the normal cone of the collision point. The boundary of the moduli space parametrises transverse point configurations on such degenerations, up to automorphisms covering the identity on $X$. This is reminiscent of tropical expansions and rubber tori, though Fulton--MacPherson's automorphism groups are typically not abelian.

Upcoming work of Mok studies configurations of distinct labelled points in the interior of a simple normal crossings pair. The resulting moduli spaces synthesise Fulton--MacPherson degenerations and tropical expansions.

\subsubsection{Unlabelled points} Another variant would be to consider configurations of unlabelled, possibly coincident points. The results of this paper would need to be modified to incorporate $S_n$-invariant structures. Special care must be taken with automorphisms, and cone stacks will likely play a role \cite[Section 2]{CCUW}. This would produce a logarithmic analogue of the punctual Chow variety.

The natural next step would be to study the Hilbert--Chow morphism to the logarithmic punctual Hilbert scheme. The latter was introduced in \cite{MR20} and studied further in \cite{KHQuot}.


\subsection*{Acknowledgements} I wish to thank Dhruv~Ranganathan, who casually suggested that it might be worth looking for a modular interpretation for the bipermutahedral variety. The tropical description of rubber tori used to characterise the boundary of the configuration spaces arose in joint work \cite{CarocciNabijouRubber} with Francesca~Carocci, whom I thank for countless discussions on this topic. I am also grateful to Luca~Battistella and Patrick Kennedy-Hunt for numerous inspiring conversations. The expert reader will clearly recognise the intellectual debt owed to the field of logarithmic and tropical enumerative geometry as a whole.

\subsection*{Competing Interests.} The author declares that there are no competing interests.

\subsection*{Funding} The author was partially supported by the Herchel Smith Fund.

\section{Scaffolds}\label{sec: scaffold}
\noindent Recall that $V[n]$ parametrises labelled points $p_0,p_1,\ldots,p_n \in V$ up to simultaneous translation. We view this as a vector space with integral structure $N[n] \subseteq V[n]$. The rigidification $x_0=1 \in T$ corresponds to the rigidification $p_0=0 \in V$ and furnishes an isomorphism $V[n] \cong V^n$. The moduli space $V[n]$ supports a universal family
\begin{equation} \label{eqn: universal family on vector spaces}
\begin{tikzcd}
V[n] \times V \ar[d,"\uppi" left] \\
V[n] \ar[u,bend right,"{p_0,p_1,\ldots,p_n}" right]
\end{tikzcd}
\end{equation}
where, letting $\uppi_i \colon V[n] \cong V^n \to V$ denote the $i$th projection, the universal sections $p_i$ are given by
\[ p_0 = \Id \times 0,\ p_1 = \Id \times \uppi_1,\  \ldots\ ,\ p_n = \Id \times \uppi_n. \]

\begin{definition} \label{def: scaffold} A \textbf{tropical scaffold} $\Lambda$ is a complete fan on $V[n] \times V$ such that the image $p_i(V[n])$ of every section $p_i$ is a union of cones in $\Lambda$.\end{definition}

For every $p \in V[n]$, intersection with the cones of $\Lambda$ gives a polyhedral decomposition of the fibre $\uppi^{-1}(p)=V$ which we denote $\Lambda_p$. The condition that the image of every section is a union of cones ensures that every point $p_i = p_i(p) \in \uppi^{-1}(p)$ belongs to a vertex of $\Lambda_p$.

\begin{remark} Tropical scaffolds are primitive cousins of the universal tropical expansions constructed for logarithmic Donaldson--Thomas theory \cite[Section~3]{MR20}. In that setting it is profitable to permit non-complete expansions. We insist on completeness in order to apply results on universal weak semistable reduction for toric stacks (Section~\ref{sec: semistable reduction}). A benefit is that the construction of the moduli space from the scaffold is canonical.
\end{remark}

\section{Moduli} \label{sec: moduli}

\noindent Given a tropical scaffold $\Lambda$ we construct the associated \textbf{configuration space} $P_\Lambda$. This is a toric Deligne--Mumford stack compactifying $T[n]$. We construct the corresponding stacky fan $\Pi_\Lambda$ on $V[n]$ explicitly from the scaffold $\Lambda$. Conceptually this is achieved by stratifying $V[n]$ into regions over which the polyhedral complex $\Lambda_p$ remains constant.

Formally, the construction $\Lambda \rightsquigarrow \Pi_\Lambda$ can be viewed either as an instance of universal weak semistable reduction \cite{AbramovichKaru,MolchoSS} or as a stacky Chow quotient \cite{KapranovSturmfelsZelevinsky,AscherMolcho}. We now explain this. However it is defined, the fan $\Pi_\Lambda$ should be such that the projection $V[n] \times V \to V[n]$ is a map of fans $\Lambda \to \Pi_\Lambda$. The corresponding toric morphism is the universal tropical expansion, so in particular should be equidimensional with reduced fibres. We construct $\Pi_\Lambda$ as the coarsest stacky fan on $V[n]$ satisfying these conditions, with the caveat that we need to allow $\Lambda$ to be replaced by a refinement (Section~\ref{sec: main construction})

After constructing $P_\Lambda$ we establish its key properties: we construct the universal family and point configuration (Section~\ref{sec: universal family}) and verify that boundary strata in $P_\Lambda$ parametrise point configurations, up to the rubber action, on the tropical expansions encoded by the scaffold (Section~\ref{sec: strata}).

While in general $P_\Lambda$ is a Deligne--Mumford stack, in most cases of interest it is a variety. The relaxed reader may therefore ignore the stacky subtleties in what follows.

\subsection{Preliminaries} We recall the basics of stacky fans and semistable reduction.

\subsubsection{Stacky fans} \label{sec: stacky fans} Toric Deligne--Mumford stacks and their associated fans appear in the literature in various incarnations \cite{BorisovChenSmith, FantechiMannNironi, GeraschenkoSatriano, GillamMolchoStackyFans}. For us, a \textbf{stacky fan} consists of an ordinary fan $\Sigma$ on a lattice $N$ together with a collection of finite-index sublattices
\[ L_\upsigma \subseteq N_\upsigma \colonequals (\upsigma \otimes \R) \cap N \]
for every $\upsigma \in \Sigma$, compatible under face inclusions. Locally, the corresponding toric stack has isotropy group $N_\upsigma/L_\upsigma$ along its deepest stratum. These were introduced in \cite{Tyomkin} and studied in \cite{GillamMolchoStackyFans} where they take the name ``lattice KM fans'' after \cite{KottkeMelrose}. A quick introduction can be found in \cite[Section~2.3]{MolchoSS}.

Stacky fans in this sense correspond to toric Deligne--Mumford stacks, possibly singular but with generically trivial isotropy. These are precisely the toric stacks which arise in our study.

\subsubsection{Weak semistability} A map of stacky fans
\[ \uppi \colon (N_1,\Sigma_1,\{L_{\upsigma_1}\}) \to (N_2,\Sigma_2,\{L_{\upsigma_2}\})\]
is \textbf{weakly semistable} if for all $\sigma_1 \in \Sigma_1$ we have
\[ \uppi(\upsigma_1) \in \Sigma_2 \quad \text{and} \quad \uppi(L_{\upsigma_1}) = L_{\uppi(\upsigma_1)}.\]
If the lattice map $\uppi \colon N_1 \to N_2$ is surjective, then $\uppi$ is weakly semistable if and only if the corresponding morphism of toric stacks is equidimensional with reduced fibres; see \cite[Sections~4 and 5]{AbramovichKaru} for the case of toric varieties, and \cite[Proposition~3.1.1]{GillamMolchoStackyFans} for the extension to toric stacks. Given a weakly semistable map of stacky fans, the corresponding morphism of toric stacks is always representable; this follows immediately from the fan criterion for representability \cite[Theorem~3.11.2]{GillamMolchoStackyFans}.

\subsubsection{Stacky Chow quotients and universal weak semistable reduction} \label{sec: semistable reduction} In \cite{AscherMolcho} Ascher--Molcho construct a stacky enhancement of the Chow quotient \cite{KapranovSturmfelsZelevinsky}. We recast this as an instance of universal weak semistable reduction \cite{MolchoSS} for maps to the logarithmic torus. We claim no originality here: we suspect that this interpretation was already known to the authors.

Fix lattices $N_1$ and $N_2$, a complete fan structure $\Sigma_1$ on $N_1$, and a surjective map
\[ \uppi \colon N_1 \to N_2.\]
Crucially, we do not begin with any fan structure on $N_2$. For the intended application, we will have $N_1=N[n] \times N$ and $N_2 = N[n]$ with $\uppi$ the projection and $\Sigma_1=\Lambda$ a tropical scaffold. It is crucial that $\Sigma_1$ is complete.

\begin{remark} We write the above data as $\uppi \colon (N_1,\Sigma_1) \to N_2$ and interpret the codomain as a ``fan'' consisting of a single cone constituting the entire lattice. If $X_1$ is the toric variety corresponding to $(N_1,\Sigma_1)$ then $\uppi$ corresponds to a logarithmic morphism $X_1 \to N_2 \otimes \mathbb{G}_{\mathrm{log}}$. See \cite{RW,MW} for background on logarithmic and tropical tori, and \cite[Section~1]{KHQuot} for a treatment of cone stacks with non-convex charts. \end{remark}

\begin{definition} The category $\mathcal{F}$ of \textbf{flattenings} has objects consisting of the following data:
\begin{enumerate}
\item $(N_2,\Sigma_2^\prime,\{L_{\upsigma_2}^\prime\})$ a stacky fan on $N_2$.\medskip
\item $(N_1,\Sigma_1^\prime,\{L_{\upsigma_1}^\prime\})$ a stacky fan refining $(N_1,\Sigma_1)$.\footnote{This means that $\Sigma_1^\prime$ is a subdivision of $\Sigma_1$. There is no condition on the sublattices $L_{\upsigma_1}^\prime$. Geometrically $\Sigma_1^\prime$ implements a toric blowup and the $L_{\upsigma_1}^\prime$ implement root constructions.}
\end{enumerate}
This data is subject to the following conditions:
\begin{enumerate}
	\item The map $N_1 \to N_2$ is a map of stacky fans $(N_1,\Sigma_1^\prime,\{L_{\upsigma_1}^\prime\}) \to (N_2,\Sigma_2^\prime,\{L_{\upsigma_2}^\prime\})$.\medskip
	\item This map is weakly semistable.
\end{enumerate}
Necessarily, the fan $\Sigma_2^\prime$ is complete. Morphisms in $\mathcal{F}$ consist of fan maps which are the identity on the underlying lattices:
\bcd
(N_1, \Sigma_1^{\prime\prime}, \{L_{\upsigma_1}^{\prime\prime}\})  \ar[r] \ar[d] & 
(N_1, \Sigma_1^{\prime}, \{L_{\upsigma_1}^{\prime}\}) \ar[r] \ar[d] & 
(N_1, \Sigma_1) \ar[d] \\
(N_2, \Sigma_2^{\prime\prime}, \{L_{\upsigma_2}^{\prime\prime}\}) \ar[r] & 
(N_2, \Sigma_2^{\prime}, \{L_{\upsigma_2}^{\prime}\}) \ar[r] & N_2.
\ecd
\end{definition}

\begin{theorem}[\!{\cite[Theorem~3.9]{AscherMolcho}}] \label{thm: semistable reduction} The category $\mathcal{F}$ has a terminal object.
\end{theorem}

The terminal object $(N_2,\Sigma_2^\prime,\{L_{\upsigma_2}^\prime\})$ is the stacky Chow quotient \cite[Section~3.2]{AscherMolcho} of $(N_1,\Sigma_1)$ by the saturated sublattice $\operatorname{Ker} \uppi \subseteq N_1$. The fan $\Sigma_2^\prime$ is obtained by overlaying the images $\uppi(\upsigma_1) \subseteq N_{2} \otimes \R$ for every cone $\upsigma_1 \in \Sigma_1$. The lattices $L_{\upsigma_2}^\prime$ are obtained by intersecting the images of the lattices $N_{\upsigma_1}$. Finally, pullback produces the stacky fan $(N_1,\Sigma_1^\prime,\{L_{\upsigma_1}^\prime\})$ refining $(N_1,\Sigma_1)$. 

\subsection{Main construction} \label{sec: main construction}

\begin{definition} \label{def: main construction} Given a tropical scaffold $\Lambda$ the associated \textbf{configuration fan} $\Pi_\Lambda$ is the stacky fan obtained by applying Theorem~\ref{thm: semistable reduction} to the projection 
\[ \uppi \colon (N[n] \times N,\Lambda) \to N[n].\]
The associated \textbf{configuration space} $P_\Lambda$ is the proper toric Deligne--Mumford stack corresponding to $\Pi_\Lambda$. This application of Theorem~\ref{thm: semistable reduction} also produces a stacky refinement of $\Lambda$ and from now on we replace $\Lambda$ by this refinement.
\end{definition}
While in general $P_\Lambda$ is a Deligne--Mumford stack, in cases of interest it is usually a variety (see Sections~\ref{sec: permutahedron} and \ref{sec: dim 2}). The necessity of stacky structures in general is explained in Example~\ref{ex: stacky example}.

\subsection{Universal family} \label{sec: universal family} Let $\Ycal_\Lambda$ denote the toric stack corresponding to $\Lambda$. The map $\Lambda \to \Pi_\Lambda$ is weakly semistable (Theorem~\ref{thm: semistable reduction}) and hence the corresponding morphism of proper toric stacks
\[ \Ycal_\Lambda \to P_\Lambda \]
is representable with equidimensional and reduced fibres. We refer to it as the \textbf{universal tropical expansion}. We now turn to the universal point configuration. Consider the section $p_i \colon V[n] \to V[n] \times V$ defined in Section~\ref{sec: scaffold}. 

\begin{lemma} The section $p_i$ is a weakly semistable map of stacky fans $\Pi_\Lambda \to \Lambda$. 
	\end{lemma}
\begin{proof} By assumption the image $p_i(V[n])$ is a union of cones of $\Lambda$ (this does not change when we refine $\Lambda$ in Definition~\ref{def: main construction}). We therefore view $p_i(V[n])$ as a subfan of $\Lambda$. The map $\uppi$ restricts to an isomorphism of vector spaces $\uppi \colon p_i(V[n]) \to V[n]$. Since $\uppi$ is weakly semistable, this restriction is an isomorphism of fans. In particular its inverse $p_i$ is also a map of fans, and maps every cone surjectively (in fact, isomorphically) onto another cone. It is straightforward to check that $p_i$ respects the stacky sublattices.
\end{proof}

Consequently, we obtain toric morphisms $x_i \colon P_\Lambda \to \Ycal_\Lambda$ which are sections of the universal tropical expansion. Together these produce the \textbf{universal point configuration} (compare with \eqref{eqn: universal family on vector spaces})
\begin{equation} \label{eqn: universal family}
\begin{tikzcd}
\Ycal_\Lambda \ar[d,"\uppi" left] \\
P_\Lambda. \ar[u,bend right,"{x_0,x_1,\ldots,x_n}" right]
\end{tikzcd}
\end{equation}
Since $p_i$ is weakly semistable, $x_i$ is torically transverse. It is not flat because the underlying lattice map is not surjective.

\subsection{Expansion geometry} \label{sec: expansion geometry} As in Section~\ref{sec: scaffold} we view the map $\Lambda \to \Pi_\Lambda$ as a family of polyhedral decompositions of $V$ parametrised by $p \in |\Pi_\Lambda| = V[n]$. The finite edges of $\Lambda_p$ are metrised by the choice of $p$, and the isomorphism class of the polyhedral complex $\Lambda_p$ is constant on the relative interior of every cone $\uprho \in \Pi_\Lambda$. Specialising to a face of $\uprho$ has the effect of setting certain edge lengths to zero, collapsing $\Lambda_p$ to a simpler polyhedral complex.

Fix a cone $\uprho \in \Pi_\Lambda$ and let $P_{\Lambda, \uprho} \subseteq P_\Lambda$ denote the corresponding locally-closed torus orbit. Let
\[ Y_\uprho \subseteq \Ycal_\Lambda \]
denote the fibre of $\uppi$ over the distinguished point of $P_{\Lambda, \uprho}$. The fibre of $\uppi$ over any other point of $P_{\Lambda, \uprho}$ is non-canonically isomorphic to $Y_\uprho$. Consider the polyhedral complex
\[ \Lambda_\uprho \colonequals \Lambda_p \]
for $p$ any point in the relative interior of $\uprho$. Polyhedra of $\Lambda_\uprho$ correspond to cones $\uplambda \in \Lambda$ with $\uppi(\uplambda) = \uprho$ and so there is an inclusion-reversing correspondence between the polyhedra of $\Lambda_\uprho$ and the strata of $Y_\uprho$. In particular the vertices of $\Lambda_\uprho$ index the irreducible components of $Y_\uprho$. Each such irreducible component is a toric variety with dense torus $T$. Its fan is obtained by zooming in to the corresponding vertex of $\Lambda_\uprho$.

Each marking $p_i$ is supported on a vertex $v_i \in \Lambda_\uprho$. This corresponds to the unique cone of $\Lambda$ which is both contained in $p_i(V[n])$ and mapped isomorphically onto $\uprho$ via $\uppi$. The corresponding irreducible component $Y_{v_i} \subseteq Y_\uprho$ is a toric variety and we denote its dense torus by $T_{v_i}$ so that
\[ T_{v_i} \subseteq Y_{v_i} \subseteq Y_\uprho.\]
There is a natural identification $T_{v_i}=T$. When the universal section $x_i \colon P_\Lambda \to \Ycal_\Lambda$ is restricted to $P_{\Lambda, \uprho}$ it factors through $T_{v_i}$. For more on the geometry of tropical expansions (in the general context of toroidal embeddings) see \cite{CarocciNabijouExpansions}.

\subsection{Rubber action and strata} \label{sec: strata}
In \cite[Section~1]{CarocciNabijouRubber} a canonical torus action is defined for any tropical expansion over a cone $\uprho$. It is referred to as the rubber action; the terminology comes from enumerative geometry. We recall the rubber action in our context.

Fix a cone $\uprho \in \Pi_\Lambda$ and a vertex $v \in \Lambda_\uprho$ indexing an irreducible component $Y_v \subseteq Y_\uprho$. There is a linear tropical position map
\[ \varphi_v \colon \uprho \to V \]
which records the position of $v$ in terms of the tropical parameters in $\uprho$. We define the \textbf{rubber torus} 
\[ T_\uprho \colonequals L_\uprho \otimes \Gm \]
where $L_\uprho \subseteq N_\uprho$ is the finite-index sublattice encoded by the stacky fan (Section~\ref{sec: stacky fans}). The position map $\varphi_v$ corresponds to a lattice map $N_\uprho \to N$ which we restrict to $L_\uprho$ and tensor by $\Gm$ to produce a homomorphism $T_\uprho \to T$. Since $Y_v$ is a toric variety with dense torus $T_v=T$ we obtain an action
\[ T_\uprho \acts Y_v.\]
These actions glue to a global action $T_\uprho \acts Y_\uprho$ referred to as the \textbf{rubber action} \cite[Theorem~1.8]{CarocciNabijouRubber}. Examples are given in \cite[Section~4]{CarocciNabijouRubber}.

Recall from Section~\ref{sec: expansion geometry} that on the locally-closed stratum $P_{\Lambda, \uprho}$ the marking $x_i$ factors through the dense torus $T_{v_i} \subseteq Y_{v_i}$. The rubber action arises from a homomorphism $T_\uprho \to T_{v_i}=T$ and hence preserves $T_{v_i}$.

\begin{theorem} \label{thm: strata} The locally-closed stratum $P_{\Lambda, \uprho}$ is isomorphic to the moduli space of point configurations
\begin{equation} \label{eqn: point configuration stratum} (x_1,\ldots,x_n) \in T_{v_1} \times \cdots \times T_{v_n} \end{equation}
considered up to the diagonal action of the rubber torus $T_{\uprho}$.
\end{theorem}
\begin{proof} For each $v_i$ the tropical position map $\varphi_{v_i} \colon L_\uprho \to N$ records the position of $p_i$. It is obtained as the composite
\[ L_\uprho \hookrightarrow N_\uprho \hookrightarrow  N[n] \xrightarrow{p_i} N[n] \times N \to N\]
where $p_i \colon N[n] \to N[n] \times N$ is the section considered in Section~\ref{sec: scaffold}. Since $p_i = \Id \times \uppi_i$ we can identify the above composite with
\[ L_\uprho \hookrightarrow N_\uprho \hookrightarrow N[n] = N^n \xrightarrow{\uppi_i} N.\]
It follows that the product of tropical position maps $\varphi_{v_1} \times \cdots \times \varphi_{v_n}$ recovers the lattice inclusion $L_{\uprho} \subseteq N[n]$. The moduli space of point configurations \eqref{eqn: point configuration stratum} is therefore the stack quotient
\[ [T[n]/T_\uprho]. \]
This is precisely the description of the torus orbit $P_{\Lambda, \uprho} \subseteq P_\Lambda$ corresponding to the cone $\uprho \in \Pi_\Lambda$.
\end{proof}

\begin{remark} The anchor point $x_0= 1 \in T$ remains stationary as the target expands. The corresponding vertex $v_0 \in \Lambda_\uprho$ is the origin in $V$ and $x_0$ is the identity element of the torus $T=T_{v_0} \subseteq Y_{v_0}$.
\end{remark}

\begin{remark}
Consider the degenerate case $\uprho=0$. The fibre $Y_0$ is the generic fibre of $\Ycal_\Lambda \to P_\Lambda$. It is the complete toric variety corresponding to the asymptotic fan of $\Lambda$ \cite[Section~3]{NishinouSiebert}. There is no rubber torus. Since a point configuration on the interior of $Y_0$ is nothing more than a point configuration on $T$, we identify the interior of $P_\Lambda$ with the moduli space $T[n]$ as expected.
\end{remark}

\begin{example}\label{ex: stacky example}
We provide an example demonstrating the necessity of stacky structures in general. Set $d=n=1$ and coordinatise the ambient space as $V=\R_x$. The position $a_1$ of the unanchored point $p_1$ coordinatises the moduli space $V[1] = \R_{a_1}$ and the projection is
\[ V[1] \times V = \R^2_{a_1 x} \xrightarrow{\uppi} \R_{a_1} = V[1].\]
We consider the following scaffold $\Lambda$:
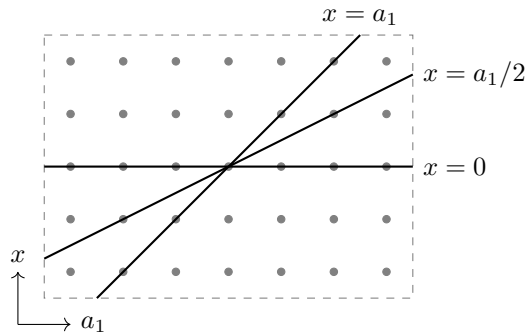
\[
\begin{tikzpicture}[scale=1.4]
\filldraw [gray] (0, 3) circle (1pt);
\filldraw [gray] (0.5, 3) circle (1pt);
\filldraw [gray] (1, 3) circle (1pt);
\filldraw [gray] (1.5, 3) circle (1pt);
\filldraw [gray] (2., 3) circle (1pt);
\filldraw [gray] (2.5, 3) circle (1pt);
\filldraw [gray] (3, 3) circle (1pt);

\filldraw [gray] (0, 2.5) circle (1pt);
\filldraw [gray] (0.5, 2.5) circle (1pt);
\filldraw [gray] (1, 2.5) circle (1pt);
\filldraw [gray] (1.5, 2.5) circle (1pt);
\filldraw [gray] (2., 2.5) circle (1pt);
\filldraw [gray] (2.5, 2.5) circle (1pt);
\filldraw [gray] (3, 2.5) circle (1pt);

\filldraw [gray] (0, 2) circle (1pt);
\filldraw [gray] (0.5, 2) circle (1pt);
\filldraw [gray] (1, 2) circle (1pt);
\filldraw [gray] (1.5, 2) circle (1pt);
\filldraw [gray] (2, 2) circle (1pt);
\filldraw [gray] (2.5, 2) circle (1pt);
\filldraw [gray] (3, 2) circle (1pt);

\filldraw [gray] (0, 1.5) circle (1pt);
\filldraw [gray] (0.5, 1.5) circle (1pt);
\filldraw [gray] (1, 1.5) circle (1pt);
\filldraw [gray] (1.5, 1.5) circle (1pt);
\filldraw [gray] (2., 1.5) circle (1pt);
\filldraw [gray] (2.5, 1.5) circle (1pt);
\filldraw [gray] (3, 1.5) circle (1pt);

\filldraw [gray] (0, 1) circle (1pt);
\filldraw [gray] (0.5, 1) circle (1pt);
\filldraw [gray] (1, 1) circle (1pt);
\filldraw [gray] (1.5, 1) circle (1pt);
\filldraw [gray] (2., 1) circle (1pt);
\filldraw [gray] (2.5, 1) circle (1pt);
\filldraw [gray] (3, 1) circle (1pt);

\draw [dashed,gray] (-0.25,3.25) -- (3.25,3.25) -- (3.25,0.75) -- (-0.25,0.75) -- (-0.25,3.25);

\draw [->] (-0.5,0.5) -- (-0.5,1);
\draw (-0.5,1) node[above]{\small$x$};
\draw [->] (-0.5,0.5) -- (0,0.5);
\draw (0,0.5) node[right]{\small$a_1$};

\draw [thick] (-0.25,2) -- (3.25,2);
\draw (3.25,2) node[right]{\small$x=0$};

\draw [thick] (0.25,0.75) --(2.75,3.25);
\draw (2.75,3.25) node[above]{\small$x=a_1$};

\draw [thick] (-0.25, 1.125) -- (3.25,2.875);
\draw (3.25,2.875) node[right]{\small$x=a_1/2$};

\end{tikzpicture}
\]
For fixed $a_1$ the corresponding vertical slice of $\Lambda$ gives a polyhedral decomposition of $V=\R_x$. For $a_1 \geq 0$ this is
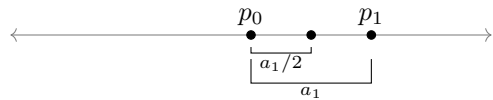
\[
\begin{tikzpicture}[scale=1.6];
\draw[<->,gray] (-2,0) -- (2,0);

\draw (0,0) node[above]{\small$p_0$};
\filldraw[black] (0,0) circle (1pt);

\filldraw[black] (0.5,0) circle (1pt);

\filldraw[black] (1,0) circle (1pt);
\draw (1,0) node[above]{\small$p_1$};

\draw (0,-0.1) -- (0,-0.15) -- (0.5,-0.15) -- (0.5,-0.1);
\draw (0.25,-0.075) node[below]{\tiny$a_1/2$};

\draw (0,-0.2) -- (0,-0.4) -- (1,-0.4) -- (1,-0.2);
\draw (0.5,-0.35) node[below]{\tiny$a_1$};

\end{tikzpicture}
\]
and similarly for $a_1 \leq 0$. The empty vertex between $p_0$ and $p_1$ corresponds to the ray $x=a_1/2$. The universal weak semistable reduction algorithm with respect to the horizontal projection produces the following stacky fan $\Pi_\Lambda$ on $V[1]=\R_{a_1}$
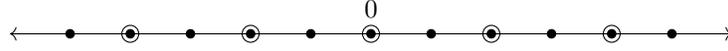
\[
\begin{tikzpicture}[scale=1.6];
\draw[<->] (-3,0) -- (3,0);

\draw (0,0.05) node[above]{\small$0$};
\filldraw[black] (0,0) circle (1pt);
\filldraw[black] (-0.5,0) circle (1pt);
\filldraw[black] (-1,0) circle (1pt);
\filldraw[black] (-1.5,0) circle (1pt);
\filldraw[black] (-2,0) circle (1pt);
\filldraw[black] (-2.5,0) circle (1pt);
\filldraw[black] (0.5,0) circle (1pt);
\filldraw[black] (1,0) circle (1pt);
\filldraw[black] (1.5,0) circle (1pt);
\filldraw[black] (2,0) circle (1pt);
\filldraw[black] (2.5,0) circle (1pt);

\draw (0,0) circle[radius=2pt];
\draw (1,0) circle[radius=2pt];
\draw (2,0) circle[radius=2pt];
\draw (-1,0) circle[radius=2pt];
\draw (-2,0) circle[radius=2pt];

\end{tikzpicture}
\]
The stacky sublattices $L_\uprho$ are indicated with open circles. The corresponding toric stack $P_\Lambda$ is the square root stack \cite{CadmanRoot} of $\PP^1$ at $0$ and $\infty$. It has two points with $\mu_2$ isotropy. Theorem~\ref{thm: strata} provides a modular interpretation for this isotropy. As $x_1$ approaches infinity the target breaks into a chain of three projective lines
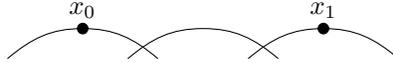
\[
\begin{tikzpicture}

\draw (0,0) to [out=40,in=180] (1,0.4) to[out=0,in=140] (2,0);
\draw[fill=black] (1,0.4) circle[radius=2pt];
\draw (1,0.4) node[above]{\small$x_0$};

\draw (1.6,0) to [out=40,in=180] (2.6,0.4) to[out=0,in=140] (3.6,0);

\draw (3.2,0) to [out=40,in=180] (4.2,0.4) to[out=0,in=140] (5.2,0);
\draw[fill=black] (4.2,0.4) circle[radius=2pt];
\draw (4.2,0.4) node[above]{\small$x_1$};
	
\end{tikzpicture}
\]
The stacky sublattice is coordinatised by $a_1/2$ and therefore the tropical position map $\Z \to \Z$ for the vertex $v_1$ is multiplication by $2$. It follows that the rubber action on the component containing $x_1$ has generic stabiliser $\mu_2$. This endows the point configuration with a nontrivial automorphism.

The above scaffold is clearly not minimal: see Question~\ref{question: minimal}.
\end{example}

\section{Dimension one: permutahedral variety} \label{sec: permutahedron} 

\noindent Set $d=1$ so that $V \cong \R$. A point configuration $(p_1,\ldots,p_n) \in V[n]$ determines a canonical polyhedral decomposition of $V$:
\[
\begin{tikzpicture}[scale=1.6];
\draw[<->] (-2,0) -- (2,0);

\draw (0,0) node[above]{\small$p_0$};
\filldraw[black] (0,0) circle (1pt);

\filldraw[black] (1,0) circle (1pt);
\draw (1,0) node[above]{\small$p_1$};

\filldraw[black] (-1.5,0) circle(1pt);
\draw (-1.5,0) node[above]{\small$p_2$};

\filldraw[black] (1.5,0) circle(1pt);
\draw (1.5,0) node[above]{\small$p_3$};
\end{tikzpicture}
\]
Consequently there is a unique minimal (with respect to refinement) tropical scaffold $\Lambda_0$. We show that the associated configuration fan $\Pi_{\Lambda_0}$ is the permutahedral fan. This recovers the identification of the Losev--Manin moduli space with the permutahedral variety \cite[Section~2.6]{LosevManin}.

\subsection{Fans from hyperplane arrangements} \label{sec: fan from hyperplane arrangement} We recall a useful general construction. Fix a lattice $N$ with associated vector space $V=N \otimes \R$. Given a pair $(a,b)$ of distinct linear functions on $N$ we obtain a hyperplane
\[ H \colonequals \{ a = b\} \subseteq V.\]
Consider a collection of such pairs $(a_1,b_1),\ldots,(a_k,b_k)$ with associated hyperplanes $H_1,\ldots,H_k$.

\begin{lemma} \label{lem: fan from hyperplane arrangement} There is a unique minimal complete fan $\Sigma$ on $N$ such that each $H_i$ is a union of cones of $\Sigma$.
\end{lemma}

\begin{proof} The idea is to subdivide $V$ into regions on which comparisons between $a_i$ and $b_i$ have been fixed for all $i$. The cones of $\Sigma$ are induced by ordered partitions
\[ I = (I_0,I_+,I_-), \qquad I_0 \sqcup I_+ \sqcup I_- = \{1,\ldots,k\}\]
with corresponding cones:
\[ \upsigma_I \colonequals \bigcap_{i \in I_0} \{a_i=b_i\} \bigcap_{i \in I_+} \{a_i \geq b_i\} \bigcap_{i \in I_-} \{a_i \leq b_i\}. \]
Each $H_i$ is the union of those cones $\sigma_I$ with $i \in I_0$, and $\Sigma$ is clearly minimal with this property. Note that different ordered partitions may induce the same cone.
\end{proof}

\subsection{Minimal scaffold} Recall from Section~\ref{sec: scaffold} the universal tropical family
\[
\begin{tikzcd}
V[n] \times V \ar[d,"\uppi" left] \\
V[n]. \ar[u,bend right,"{p_0,p_1,\ldots,p_n}" right]
\end{tikzcd}
\]
The key point is that for $d=1$, the image of each section $p_i$ is a hyperplane. We write
\[ H_i \colonequals p_i(V[n]).\]
The identification $V[n]=V^n$ furnishes coordinates $a_1,\ldots,a_n$ on $V[n]$; by convention we also set $a_0=0$. Let $x$ denote the standard coordinate on $V$. The hyperplane $H_i$ is then given by
\[ H_i = \{x = a_i\}.\]
Recall that a tropical scaffold is a complete fan on $V[n] \times V$ such that each $H_i$ is a union of cones.

\begin{lemma} \label{lem: unique minimal scaffold} There is a unique minimal tropical scaffold $\Lambda_0$.	
\end{lemma}

\begin{proof} This follows from Lemma~\ref{lem: fan from hyperplane arrangement}. Explicitly, the cones of $\Lambda_0$ are indexed by ordered partitions $I = (I_0,I_+,I_-)$ where $I_0 \sqcup I_+ \sqcup I_- = \{0,1,\ldots,n\}$, and the corresponding cones are:
\begin{equation} \label{eqn: definition of lambdaI} \uplambda_I \colonequals \bigcap_{i \in I_0} \{ x = a_i \} \bigcap_{i \in I_+} \{ x \geq  a_i \} \bigcap_{i \in I_-} \{ x \leq a_i \}. \qedhere \end{equation}
\end{proof}

\subsection{Permutahedral fan} Let $\Sigma_n$ denote the permutahedral fan on $V[n]$. Conceptually this stratifies $V[n]$ into regions over which the functions $a_0,a_1,\ldots,a_n$ have a fixed ordering. Geometrically this corresponds to a fixed ordering of the points $p_0,p_1,\ldots,p_n$. 

Formally the cones of $\Sigma_n$ are indexed by total preorders on $\{0,1,\ldots,n\}$. We encode these as ordered partitions
\[ J=(J_1, \ldots, J_m), \qquad J_i \neq \emptyset,\ J_1 \sqcup \ldots \sqcup J_m = \{0,1,\ldots,n\}. \]
The corresponding cones are
\begin{equation} \label{eqn: definition of deltaJ} \upsigma_J \colonequals \bigcap_{\substack{1 \leq i \leq m \\ a_k,a_l \in J_i}} \{ a_k=a_l\} \bigcap_{\substack{1 \leq i < j \leq m \\ a_i \in J_i, a_j \in J_j}} \{ a_i \leq a_j \}. \end{equation}
\begin{theorem} \label{thm: get permutahedron} The configuration fan $\Pi_{\Lambda_0}$ is equal to the permutahedral fan $\Sigma_n$.	
\end{theorem}

\begin{proof} We first show that $\Pi_{\Lambda_0}$ is a refinement of $\Sigma_n$. It suffices to show that every cone of $\Sigma_n$ is an intersection of images of cones of $\Lambda_0$. Fix $\upsigma_J \in \Sigma_n$ corresponding to the total preorder $J=(J_1,\ldots,J_m)$. For each $i \in \{1,\ldots,m\}$ we define the partition $I(i) = (I_0(i),I_+(i),I_-(i))$ by:
\begin{align*}
	I_0(i) & \colonequals J_i, \\
	I_+(i) & \colonequals J_1 \sqcup \ldots \sqcup J_{i-1}, \\
	I_-(i) & \colonequals J_{i+1} \sqcup \ldots \sqcup J_m.	\end{align*}
By \eqref{eqn: definition of lambdaI} on the corresponding cone $\uplambda_{I(i)} \in \Lambda_0$ we have
\begin{align*} x & = a_j \qquad \text{for $j \in J_i$},\\
x & \geq a_j \qquad \text{for $j \in J_1 \sqcup \ldots \sqcup J_{i-1}$},	\\
x & \leq a_j \qquad \text{for $j \in J_{i+1} \sqcup \ldots \sqcup J_m$}.
\end{align*}
Projecting away from the coordinate $x$ we obtain
\[ \uppi(\uplambda_{I(i)}) = \bigcap_{a_k,a_l \in J_i} \{ a_k = a_l \} \bigcap_{\substack{1 \leq j < i\\ a_j \in J_j\\ a_i \in J_i}} \{ a_j \leq a_i \} \bigcap_{\substack{i < j \leq m\\ a_i \in J_i\\ a_j \in J_j}} \{ a_i \leq a_j \}.\]
It follows from \eqref{eqn: definition of deltaJ} that $\uppi(\uplambda_{I(1)}) \cap \cdots \cap \uppi(\uplambda_{I(m)}) = \upsigma_J$ (in general there is no single cone $\uplambda \in \Lambda_0$ whose image is $\upsigma_J$). We conclude that there is a refinement
\[ \Pi_{\Lambda_0} \to \Sigma_n.\]
To prove that this is an equality, we invoke the universal property of $\Pi_{\Lambda_0}$ (Theorem~\ref{thm: semistable reduction}). It suffices to construct a refinement $\Lambda_0^\prime \to \Lambda_0$ such that $\uppi$ is a weakly semistable map of fans $\Lambda_0^\prime \to \Sigma_n$.

The fan $\Lambda_0^\prime$ is constructed as the minimal common refinement of $\Lambda_0$ and the preimage of $\Sigma_n$ under $\uppi$. Since both $\Lambda_0$ and $\Sigma_n$ are defined by hyperplanes, this can be described explicitly. We see from \eqref{eqn: definition of lambdaI} that the hyperplanes defining $\Lambda_0$ impose comparisons between each of $a_0,a_1,\ldots,a_n$ and $x$. The hyperplanes defining $\Sigma_n$ imposes pairwise comparisons amongst $a_0,a_1,\ldots,a_n$, and the same is true of their pullbacks under $\uppi$. The union of these hyperplanes therefore impose pairwise comparisons amongst $a_0,a_1,\ldots,a_n,x$. We conclude that
\[ \Lambda_0^\prime = \Sigma_{n+1} \]
with the coordinate $x$ playing the role of $a_{n+1}$. It is well-known and not hard to check that projecting away from this coordinate gives a weakly semistable map of fans $\Sigma_{n+1} \to \Sigma_n$. This completes the proof.
\end{proof}

The configuration space $P_{\Lambda_0}$ is therefore identified with the Losev--Manin moduli space. The universal tropical expansion and point configuration coincide with the universal curve and marking sections
\[
\begin{tikzcd}
\Ycal_{\Lambda_0^\prime} \ar[d,"\uppi" left] \\
P_{\Lambda_0}. \ar[u,bend right,"{x_0,x_1,\ldots,x_n}" right]
\end{tikzcd}
\]
This follows from the identification $\Lambda_0^\prime=\Sigma_{n+1}$ and the fact that the forgetful map $\Sigma_{n+1} \to \Sigma_n$ induces the universal curve over the Losev--Manin moduli space \cite[Section~2.1]{LosevManin}.

The rubber action on the universal expansion coincides with the automorphisms of the universal Losev--Manin curve which fix the two heavy markings and $x_0$. This is illustrated in the following example.

\begin{example} Consider the following tropical point configuration with $n=3$:
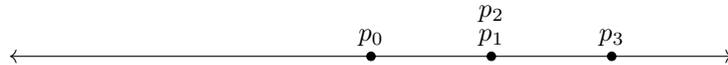
\[
\begin{tikzpicture}[scale=1.6];
\draw[<->] (-3,0) -- (3,0);

\draw (0,0) node[above]{\small$p_0$};
\filldraw[black] (0,0) circle (1pt);

\filldraw[black] (1,0) circle (1pt);
\draw (1,0) node[above]{\small$p_1$};
\draw (1,0.2) node[above]{\small$p_2$};

\filldraw[black] (2,0) circle(1pt);
\draw (2,0) node[above]{\small$p_3$};

\end{tikzpicture}
\]
This defines a total preorder $J = (J_1,J_2,J_3) = (\{0\},\{1,2\},\{3\})$ indexing a two-dimensional cone $\upsigma_J \in \Pi_{\Lambda_0}=\Sigma_3$. This cone is defined by the (in)equalities
\[ a_0 \leq a_1 = a_2 \leq a_3.\]
It follows that the primitive positive coordinates on the cone $\upsigma_J$ are
\[ a_1-a_0, \qquad a_3-a_1\]
which appear geometrically as the edge lengths in the polyhedral decomposition
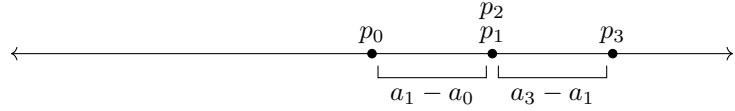
\[
\begin{tikzpicture}[scale=1.6];
\draw[<->] (-3,0) -- (3,0);

\draw (0,0) node[above]{\small$p_0$};
\filldraw[black] (0,0) circle (1pt);

\filldraw[black] (1,0) circle (1pt);
\draw (1,0) node[above]{\small$p_1$};
\draw (1,0.2) node[above]{\small$p_2$};

\filldraw[black] (2,0) circle(1pt);
\draw (2,0) node[above]{\small$p_3$};

\draw (0.05,-0.1) -- (0.05,-0.2) -- (0.95,-0.2) -- (0.95,-0.1);
\draw (0.5,-0.175) node[below]{\small$a_1-a_0$};

\draw (1.05,-0.1) -- (1.05,-0.2) -- (1.95,-0.2) -- (1.95,-0.1);
\draw (1.5,-0.175) node[below]{\small$a_3-a_1$};

\end{tikzpicture}
\]
The corresponding topical expansion is a chain of three projective lines
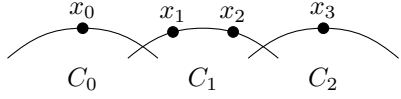
\[
\begin{tikzpicture}

\draw (0,0) to [out=40,in=180] (1,0.4) to[out=0,in=140] (2,0);
\draw[fill=black] (1,0.4) circle[radius=2pt];
\draw (1,0.4) node[above]{\small$x_0$};
\draw (1,0) node[below]{\small$C_0$};

\draw (1.6,0) to [out=40,in=180] (2.6,0.4) to[out=0,in=140] (3.6,0);
\draw (2.6,0) node[below]{\small$C_1$};
\draw[fill=black] (2.2,0.35) circle[radius=2pt];
\draw (2.2,0.35) node[above]{\small$x_1$};
\draw[fill=black] (3,0.35) circle[radius=2pt];
\draw (3,0.35) node[above]{\small$x_2$};

\draw (3.2,0) to [out=40,in=180] (4.2,0.4) to[out=0,in=140] (5.2,0);
\draw[fill=black] (4.2,0.4) circle[radius=2pt];
\draw (4.2,0.4) node[above]{\small$x_3$};
\draw (4.2,0) node[below]{\small$C_2$};
	
\end{tikzpicture}
\]
The two-dimensional rubber torus $T_{\upsigma_J}=N_{\upsigma_J} \otimes \Gm$ is canonically coordinatised by $e_1 \colonequals a_1-a_0$ and $e_2 \colonequals a_3-a_1$. With respect to these coordinates the action on the components $C_1$ and $C_2$ has weights $e_1$ and $e_1+e_2$ respectively. In the latter case this is because the position of the corresponding vertex is $e_1+e_2=a_3-a_0$.

 Changing coordinates from $(e_1,e_2)$ to $(e_1,e_1+e_2)$ produces a split torus whose factors act independently on the components $C_1$ and $C_2$. This coincides with the automorphisms of the curve obtained by forgetting the markings $x_1,x_2,x_3$ and introducing a heavy marking on each end component, as in the Losev--Manin moduli space. The rigidification $x_0=1 \in T$ means that this marking does not contribute to moduli and is not forgotten. This is consistent with the fact that the rubber torus acts trivially on $C_0$.
\end{example}

\section{Dimension two: bipermutahedral variety} \label{sec: dim 2}

\noindent For $d \geq 2$ there is no longer a unique minimal choice of scaffold. We introduce two scaffolds of particular interest, and show that the associated configuration spaces are the square of the permutahedral variety (Section~\ref{sec: square of permutahedron}) and the bipermutahedral variety (Section~\ref{sec: bipermutahedron}).

\subsection{Coordinates} Set $d=2$. We coordinatise the ambient space as
\[ V = \R^2_{xy} \]
which also produces coordinates on the configuration space
\[ V[n] = \R^{2n}_{a_1 b_1 \ldots a_n b_n}\]
where $(a_i,b_i)$ is the position of $p_i$. The anchor point $p_0$ has coordinates $a_0=b_0=0$.

\subsection{Square of the permutahedral fan} \label{sec: square of permutahedron} Given a point configuration $(p_1,\ldots,p_n) \in V[n]$ we obtain a polyhedral decomposition of $V$ by slicing with the horizontal and vertical lines passing through the points $p_i$
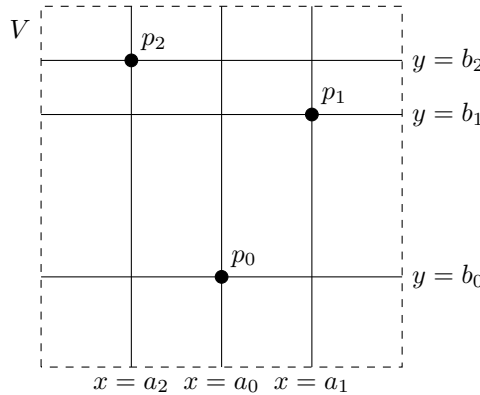
\[
\begin{tikzpicture}[scale=1.2]

\draw (5,1.75) node[left]{\small$V$};
\draw [dashed] (5,-2) -- (5,2) -- (9,2) -- (9,-2) -- (5,-2);

\draw (5,-1) -- (9,-1);
\draw(9,-1) node[right]{\small$y=b_0$};

\draw (5,0.8) -- (9,0.8);
\draw (9,0.8) node[right]{\small$y=b_1$};

\draw (5,1.4) -- (9,1.4);
\draw (9,1.4) node[right]{\small$y=b_2$};

\draw (6,2) -- (6,-2);
\draw (6,-2) node[below]{\small$x=a_2$};

\draw (7,2) -- (7,-2);
\draw (7,-2) node[below]{\small$x=a_0$};

\draw (8,2) -- (8,-2);
\draw (8,-2) node[below]{\small$x=a_1$};

\draw (7.25,-1) node[above]{\small$p_0$};
\draw[fill=black] (7,-1) circle[radius=2pt];

\draw (8.25,0.8) node[above]{\small$p_1$};
\draw[fill=black] (8,0.8) circle[radius=2pt];

\draw (6.25,1.4) node[above]{\small$p_2$};
\draw[fill=black] (6,1.4) circle[radius=2pt];

\end{tikzpicture}
\]
This defines a tropical scaffold $\Lambda_{S}$ ($S$ for ``square''). The horizontal and vertical lines above correspond to hyperplanes in $V[n] \times V$:

\begin{definition} \label{defn: scaffold square of permutahedron} The scaffold $\Lambda_{S}$ is the complete fan on $V[n] \times V$ induced (as in Section~\ref{sec: fan from hyperplane arrangement}) by the following arrangement of hyperplanes:
\begin{equation} \label{eqn: hyperplanes for LambdaS} \{ x = a_0 \} , \ldots, \{ x\, = a_n\}, \{ y = b_0 \}, \ldots, \{ y = b_n\}. \end{equation}
\end{definition}
The polyhedral complex $\Lambda_{S,p}$ for $p = (a_1,b_1,\ldots,a_n,b_n) \in V[n]$ undergoes phase transitions whenever the ordering of $(a_0,\ldots,a_n)$ or $(b_0,\ldots,b_n)$ changes. This explains the following.

\begin{proposition}
	The configuration fan $\Pi_{\Lambda_{S}}$ is equal to the square $\Sigma_n \times \Sigma_n$ of the permutahedral fan.
\end{proposition}

\begin{proof} There is a natural isomorphism
\[ \R^{2n}_{a_1 b_1 \ldots a_n b_n} \times \R^2_{xy} = (\R^n_{a_1 \ldots a_n} \times \R_x) \times (\R^n_{b_1 \ldots b_n} \times \R_y).\]
Under this isomorphism we have $\Lambda_{S} = \Lambda_0 \times \Lambda_0$ (compare Definition~\ref{defn: scaffold square of permutahedron} and the proof of Lemma~\ref{lem: unique minimal scaffold}). The result follows from Theorem~\ref{thm: get permutahedron}, since universal weak semistable reduction commutes with external products.
\end{proof}

\begin{example} \label{example: square of permutahedron} Consider the point configuration and scaffold in Figure~\ref{fig: square scaffold}. The corresponding cone $\uprho \in \Pi_{\Lambda_{S}} = \Sigma_n \times \Sigma_n$ is defined by the inequalities
\begin{equation} \label{eqn: inequalities defining square cone example} a_1 \leq a_0 \leq a_2,\qquad  b_0 \leq b_2 \leq b_1. \end{equation}
The fibre of the universal tropical expansion \eqref{eqn: universal family} over the locally-closed stratum $P_{\Lambda_{S},\uprho} \subseteq P_{\Lambda_S}$ is the patchwork quilt made of copies of $\PP^1\!\times\!\PP^1$ depicted in Figure~\ref{fig: square expansion}.
\begin{figure}[h]
\centering
\begin{subfigure}[b]{0.3\textwidth}
\[
\begin{tikzpicture}[scale=1]

\draw [dashed] (0,0) -- (5,0) -- (5,4.5) -- (0,4.5) -- (0,0);

\draw (0,1) -- (5,1);
\draw (1,0) -- (1,4.5);
\draw (0,3.5) -- (5,3.5);
\draw (4.5,0) -- (4.5,4.5);
\draw (0,2.5) -- (5,2.5);
\draw (3,0) -- (3,4.5);

\draw [decorate,decoration={brace,amplitude=5pt,mirror,raise=0.25ex}] (1.1,1) -- (2.9,1) node[midway,yshift=-11.5pt]{\small$e_1$};
\draw [decorate,decoration={brace,amplitude=5pt,mirror,raise=0.25ex}] (3.1,1) -- (4.4,1) node[midway,yshift=-11.5pt]{\small$e_2$};
\draw [decorate,decoration={brace,amplitude=5pt}] (0.95,1.1) -- (0.95,2.4) node[midway,xshift=-10pt]{\small$f_1$};
\draw [decorate,decoration={brace,amplitude=5pt}] (0.95,2.6) -- (0.95,3.4) node[midway,xshift=-10pt]{\small$f_2$};

\draw (3.25,1) node[above]{\small$p_0$};
\draw[fill=black] (3,1) circle[radius=2pt];

\draw (1.25,3.5) node[above]{\small$p_1$};
\draw[fill=black] (1,3.5) circle[radius=2pt];

\draw (4.25,2.5) node[above]{\small$p_2$};
\draw[fill=black] (4.5,2.5) circle[radius=2pt];

\end{tikzpicture}
\]
\caption{Scaffold}
\label{fig: square scaffold}
\end{subfigure}\qquad\qquad
\begin{subfigure}[b]{0.3\textwidth}
\[
\begin{tikzpicture}[scale=1.2]

\draw (0,0) -- (1,0) -- (1,1) -- (0,1) -- (0,0);
\draw (1,0) -- (2,0) -- (2,1) -- (1,1) -- (1,0);
\draw (2,0) -- (3,0) -- (3,1) -- (2,1) -- (2,0);

\draw (0,1) -- (1,1) -- (1,2) -- (0,2) -- (0,1);
\draw (1,1) -- (2,1) -- (2,2) -- (1,2) -- (1,1);
\draw (2,1) -- (3,1) -- (3,2) -- (2,2) -- (2,1);

\draw (0,2) -- (1,2) -- (1,3) -- (0,3) -- (0,2);
\draw (1,2) -- (2,2) -- (2,3) -- (1,3) -- (1,2);
\draw (2,2) -- (3,2) -- (3,3) -- (2,3) -- (2,2);

\draw[fill=black] (1.5,0.5) circle[radius=1.75pt];
\draw (1.55,0.5) node[below]{\small$x_0$};

\draw[fill=black] (0.6,2.6) circle[radius=1.75pt];
\draw (0.65,2.6) node[below]{\small$x_1$};

\draw[fill=black] (2.4,1.6) circle[radius=1.75pt];
\draw (2.45,1.6) node[below]{\small$x_2$};

\end{tikzpicture}
\]
\caption{Expansion}
\label{fig: square expansion}
\end{subfigure}
\caption{A stratum of $P_{\Lambda_S}$.}
\end{figure}
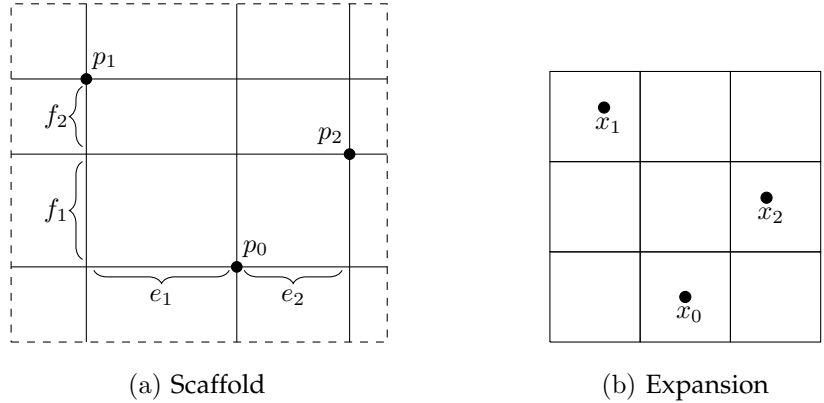

The polyhedral edge lengths labelled in Figure~\ref{fig: square scaffold} give the coordinate system on $\uprho$ dual to the basis of primitive ray generators. Indeed, the inequalities \eqref{eqn: inequalities defining square cone example} produce the primitive non-negative linear functions
\[ a_0 - a_1  = e_1, \qquad a_2 - a_0 = e_2, \qquad b_2 - b_0 = f_1, \qquad b_1 - b_2 = f_2.\]
We use Theorem~\ref{thm: strata} to give a modular description of $P_{\Lambda_S,\uprho}$. The rubber torus is $T_{\uprho} = (\Gm^4)_{e_1 e_2 f_1 f_2}$ and the rubber action is governed by the tropical position maps. For the vertices $v_i$ supporting the points $p_i$ these are
\begin{align*} 
\varphi_{v_0}(e_1,e_2,f_1,f_2) & = (a_0,b_0) = (0,0), \\
\varphi_{v_1}(e_1,e_2,f_1,f_2) & = (a_1,b_1) = (-e_1,f_1+f_2), \\
\varphi_{v_2}(e_1,e_2,f_1,f_2) & = (a_2,b_2) = (e_2,f_1).
\end{align*}
The positions of $x_1,x_2$ give $2+2=4$ dimensions of moduli. These are precisely cancelled out by the rubber action, so $P_{\Lambda_S,\uprho}$ is a point. This is consistent with the fact that $\uprho \in \Pi_{\Lambda_S}$ is maximal.
\end{example}

\subsection{Bipermutahedral fan} \label{sec: bipermutahedron}

We assume familiarity with the bipermutahedral fan. For an accessible introduction, see \cite[Sections~2.3--2.6]{ArdilaDenhamHuh} and \cite{ArdilaBipermutahedron}.

Ardila--Denham--Huh define the bipermutahedral fan $\Sigma_{n,n}$ as the tropical configuration space $V[n]$ stratified according to bisequence \cite[Section~2.4]{ArdilaDenhamHuh}. From our perspective the definition of bisequence suggests a choice of tropical scaffold, namely the subdivision
\[ \Lambda_B \to \Lambda_S \]
obtained by slicing each polyhedral decomposition $\Lambda_{S,p}$ with the supporting antidiagonal (compare with \cite[Figure~2]{ArdilaDenhamHuh}):
\begin{equation} \label{eqn: bipermutahedral scaffold} 
\begin{tikzpicture}[scale=1,baseline=(current  bounding  box.center)]

\draw [dashed] (5,-2) -- (5,2) -- (9,2) -- (9,-2) -- (5,-2);

\draw[blue,thick] (8,-2) -- (5,1);
\draw[blue] (8,-2) node[below]{\small$x+y=a_0+b_0$};

\draw (5,-1) -- (9,-1);

\draw (5,0.8) -- (9,0.8);

\draw (5,1.4) -- (9,1.4);

\draw (6,2) -- (6,-2);

\draw (7,2) -- (7,-2);

\draw (8,2) -- (8,-2);

\draw (7.25,-1) node[above]{\small$p_0$};
\draw[fill=black] (7,-1) circle[radius=2pt];

\draw (8.25,0.8) node[above]{\small$p_1$};
\draw[fill=black] (8,0.8) circle[radius=2pt];

\draw (6.25,1.4) node[above]{\small$p_2$};
\draw[fill=black] (6,1.4) circle[radius=2pt];

\end{tikzpicture}
\end{equation}
Formally this is defined as follows. The vector space $V[n] \times V$ is covered by closed convex cones
\[ \mathcal{C}_i \colonequals \left\{ \min_{j}(a_j + b_j) = a_i + b_i \right\} \subseteq V[n] \times V \]
for $i \in \{0,\ldots,n\}$. Geometrically $\mathcal{C}_i$ is the locus where $p_i$ lies on the supporting antidiagonal; over $\mathcal{C}_i$ the supporting antidiagonal thus has equation $x+y=a_i+b_i$.

The restricted fan $\Lambda_{S}|_{\mathcal{C}_i}$ is defined by intersecting each cone of $\Lambda_S$ with $\mathcal{C}_i$. We construct $\Lambda_B|_{\mathcal{C}_i}$ by slicing $\Lambda_S|_{\mathcal{C}_i}$ with the hyperplane
\[ D_i \colonequals \{ x+y = a_i+b_i \} \subseteq V[n] \times V\]
so that $\Lambda_B|_{\mathcal{C}_i}$ is the restriction to $\mathcal{C}_i$ of the fan induced by the arrangement of hyperplanes \eqref{eqn: hyperplanes for LambdaS} together with $D_i$ (as in Section~\ref{sec: fan from hyperplane arrangement}). Since
\[ D_i \cap (\mathcal{C}_i \cap \mathcal{C}_j) = D_j \cap (\mathcal{C}_i \cap \mathcal{C}_j) \]
the fans $\Lambda_B|_{\mathcal{C}_i}$ glue to produce a scaffold $\Lambda_B$. Unlike $\Lambda_0$ or $\Lambda_S$, $\Lambda_B$ is not globally induced by a hyperplane arrangement.

The fibres $\Lambda_{B,p}$ are depicted in \eqref{eqn: bipermutahedral scaffold}. This polyhedral complex undergoes phase transitions precisely when the bisequence associated to $p \in V[n]$ changes. This explains the following.

\begin{theorem} \label{thm: get bipermutahedron} The configuration fan $\Pi_{\Lambda_B}$ is equal to the bipermutahedral fan $\Sigma_{n,n}$.	
\end{theorem}

\begin{proof} We first describe $\Sigma_{n,n}$ in the form we require. The vector space $V[n]$ is covered by closed convex cones
\[ C_i \colonequals \left\{ \min_j(a_j+b_j) = a_i + b_i \right\} \subseteq V[n] \]
for $i \in \{0,\ldots,n\}$. Fixing $i$, the fan $\Sigma_{n,n}|_{C_i}$ is constructed as in Section~\ref{sec: fan from hyperplane arrangement} by imposing comparisons for the following pairs of linear functions:
\begin{align*} & (a_j,a_k) \qquad \qquad \quad 0 \leq j < k \leq n,\\
& (b_j,b_k) \qquad \qquad \quad\, 0 \leq j < k \leq n, \\
& (a_j+b_k,a_i+b_i) \qquad 0 \leq j,k \leq n.
\end{align*}
Equivalently, each cone $\upsigma \in \Sigma_{n,n}|_{C_i}$ is obtained by choosing the following comparisons:\footnote{The comparisons of type (i), (ii), (iii) contain redundancies. Since $\upsigma \subseteq C_i$ we must have $a_i+b_i \leq a_j+b_j$ for all $j$. Consequently, choosing $a_i+b_i = a_j+b_j$ or $a_i+b_i \geq a_j+b_j$ in (iii) produces the same cone. Similarly if we choose $a_i \leq a_j$ in (i) and $b_i \leq b_k$ in (ii) then we automatically have $a_i+b_i \leq a_j+b_k$, and so choosing $a_i+b_i = a_j+b_k$ or $a_i+b_i \geq a_j+b_k$ in (iii) produces the same cone. These redundancies do not affect the argument.}
\begin{enumerate}[label=(\roman*)]
\item A total preorder of $a_0,\ldots,a_n$.
\item A total preorder of $b_0,\ldots,b_n$.
\item A comparison of $a_j+b_k$ against $a_i+b_i$ for $0 \leq j,k \leq n$.
\end{enumerate}
Together these comparisons determine the bisequence associated to a point $p=(a_1,b_1,\ldots,a_n,b_n) \in C_i$. The cone $\upsigma \subseteq C_i$ is cut out by the corresponding (in)equalities. The fans $\Sigma_{n,n}|_{C_i}$ glue to produce $\Sigma_{n,n}$. The comparisons of type (i) and (ii) show that $\Sigma_{n,n}$ is a refinement of $\Sigma_n \times \Sigma_n$.

To show that $\Pi_{\Lambda_B}=\Sigma_{n,n}$ we follow the proof of Theorem~\ref{thm: get permutahedron}. We first prove that $\Pi_{\Lambda_B}$ is a refinement of $\Sigma_{n,n}$, by showing that each cone of $\Sigma_{n,n}$ is an intersection of images of cones of $\Lambda_B$.

Fix a cone $\upsigma \in \Sigma_{n,n}$. We have $\upsigma \in \Sigma_{n,n}|_{C_i}$ for some $i$ (not necessarily unique) and $\upsigma$ is cut out by the (in)equalities imposing chosen comparisons of type (i), (ii), (iii). For each of these (in)equalities we find a cone $\uplambda \in \Lambda_B$ such that the given (in)equality holds on $\uppi(\uplambda)$, and such that $\upsigma \subseteq \uppi(\uplambda)$. The intersection of all such $\uppi(\uplambda)$ is then equal to $\upsigma$.

Since $\uppi^{-1}(C_i) = \mathcal{C}_i$ we must have $\uplambda \in \Lambda_B|_{\mathcal{C}_i}$. A cone $\uplambda \in \Lambda_B|_{\mathcal{C}_i}$ is obtained by choosing the following comparisons:
\begin{enumerate}[label=(\Roman*)]
	\item A comparison of $x$ against each of $a_0,\ldots,a_n$.
	\item A comparison of $y$ against each of $b_0,\ldots,b_n$.
	\item A comparison of $x+y$ against $a_i+b_i$.
\end{enumerate}
We work through the comparisons defining $\upsigma$. First consider a comparison of type (i), comparing a pair $(a_j,a_k)$. Choose the type (I) comparison $x=a_j$. For the other $a_l$ (including $l=k$) choose the type (I) comparison for the pair $(x,a_l)$ to be equal to the type (i) comparison for the pair $(a_j,a_l)$. This specifies the type (I) comparisons. For the type (II) comparisons, choose $y=b_0$ and for $1 \leq l \leq n$ choose the comparison for $(y,b_l)$ to be the same as the type (i) comparison for $(b_0,b_l)$. This specifies the type (II) comparisons. Finally for the type (III) comparison, choose the comparison for $(x+y,a_i+b_i)$ to be the same as the type (iii) comparison for $(a_j+b_0,a_i+b_i)$. This produces comparisons of types (I), (II), (III) and hence a cone $\uplambda \in \Lambda_B|_{\Ccal_i}$. By construction the type (i) comparison for the pair $(a_j,a_k)$ holds on $\uppi(\lambda)$, and we have $\upsigma \subseteq \uppi(\uplambda)$ as required.

The case of a type (ii) comparison is identical. It remains to consider a type (iii) comparison. This compares $a_j+b_k$ against $a_i+b_i$. We choose the type (I) and (II) comparisons
\[ x = a_j, \qquad y=b_k, \]
and choose the type (III) comparison for $(x+y,a_i+b_i)$ to be the same as the type (iii) comparison for $(a_j+b_k,a_i+b_i)$. We choose the remaining type (I) and (II) comparisons compatibly with the type (i) and (ii) comparisons, as in the previous case. This produces a cone $\uplambda \in \Lambda_B|_{\mathcal{C}_i}$ such that the type (iii) comparison for the pair $(a_j+b_k,a_i+b_i)$ holds on $\uppi(\uplambda)$, and $\upsigma \subseteq \uppi(\uplambda)$ as required.

We conclude that there is a refinement
\[ \Pi_{\Lambda_B} \to \Sigma_{n,n}.\]
To prove that this is an equality, we invoke the universal property of $\Pi_{\Lambda_B}$ (Theorem~\ref{thm: semistable reduction}). It suffices to construct a refinement $\Lambda_B^\prime \to \Lambda_B$ such that $\uppi$ is a weakly semistable map of fans $\Lambda_B^\prime \to \Sigma_{n,n}$.

We construct $\Lambda_B^\prime$ as the minimal common refinement of $\Lambda_B$ and the preimage of $\Sigma_{n,n}$ under $\uppi$. The condition of being a weakly semistable map of fans is local on the source, so it suffices to show that $\uppi$ is a weakly semistable map of fans
\begin{equation} \label{eqn: map lambda prime to sigma} \Lambda_B^\prime|_{\mathcal{C}_i} \to \Sigma_{n,n}|_{C_i}.\end{equation}
Both $\Lambda_B|_{\mathcal{C}_i}$ and $\Sigma_{n,n}|_{C_i}$ are induced by hyperplane arrangements. We deduce that a cone $\uplambda^\prime \in \Lambda_B^\prime|_{\mathcal{C}_i}$ is obtained by choosing the following comparisons:
\begin{enumerate}[label=(\Roman*')]
\item A total preorder of $x,a_0,\ldots,a_n$.
\item A total preorder of $y,b_0,\ldots,b_n$.
\item A comparison of $a_j+b_k$ against $a_i+b_i$ for $0 \leq j,k \leq n$.
\item A comparison of $x+y$ against $a_i+b_i$.
\end{enumerate}
Fix a cone $\uplambda^\prime \in \Lambda^\prime_B|_{\Ccal_i}$. The comparisons of type (I') and (II') immediately determine comparisons of type (i) and (ii). We produce comparisons of type (iii) by combining the comparisons of type (III') with an additional set of comparisons, obtained as follows. If $x \geq a_j$ in (I'), $y \geq b_k$ in (II'), and $x+y \geq a_i+b_i$ in (IV'), then we append the type (iii) comparison $a_j+b_k \geq a_i+b_i$ (and similarly with the inequalities reversed).

We have thus produced a cone $\upsigma \in \Sigma_{n,n}|_{C_i}$ with $\uppi(\uplambda^\prime) \subseteq \upsigma$. This shows that \eqref{eqn: map lambda prime to sigma} is a map of fans. It remains to show that it is weakly semistable.

Fix $p=(a_1,b_1,\ldots,a_n,b_n) \in |\upsigma|$. The type (I'), (II'), (IV') comparisons that involve $x$ and $y$ do not impose any restrictions on the $a_j$ and $b_j$ beyond the type (i), (ii), (iii) comparisons constructed above. It follows that $p$ admits a lift, and hence $\uppi(\uplambda^\prime)=\upsigma$ as required.
\end{proof}

\begin{example} Consider the cone $\uprho \in \Sigma_n \times \Sigma_n$ from Example~\ref{example: square of permutahedron}. This is cut out by the inequalities \eqref{eqn: inequalities defining square cone example}. These correspond to the type (i) and (ii) comparisons appearing in the proof of Theorem~\ref{thm: get bipermutahedron}. 

The refinement $\Sigma_{n,n} \to \Sigma_n \times \Sigma_n$ subdivides $\uprho$ into a union of cones. By Theorem~\ref{thm: get bipermutahedron} each such cone corresponds to a different polyhedral complex $\Lambda_{B,p}$. An example is illustrated in Figure~\ref{fig: biperm scaffold}. 
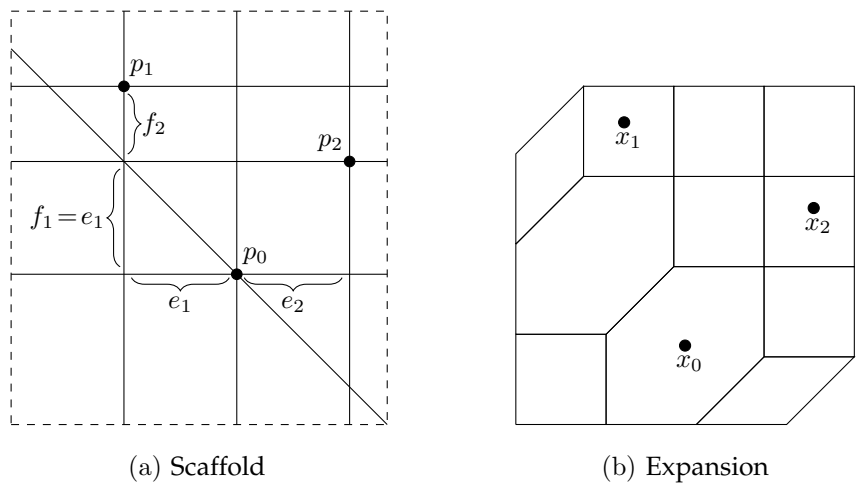
\begin{figure}[h]
\centering
\begin{subfigure}[b]{0.3\textwidth}
\[
\begin{tikzpicture}[scale=1]

\draw [dashed] (0,-1) -- (5,-1) -- (5,4.5) -- (0,4.5) -- (0,-1);

\draw (0,1) -- (5,1);
\draw (1.5,-1) -- (1.5,4.5);
\draw (0,3.5) -- (5,3.5);
\draw (4.5,-1) -- (4.5,4.5);
\draw (0,2.5) -- (5,2.5);
\draw (3,-1) -- (3,4.5);

\draw [decorate,decoration={brace,amplitude=5pt,mirror,raise=0.25ex}] (1.6,1) -- (2.9,1) node[midway,yshift=-11.5pt]{\small$e_1$};
\draw [decorate,decoration={brace,amplitude=5pt,mirror,raise=0.25ex}] (3.1,1) -- (4.4,1) node[midway,yshift=-11.5pt]{\small$e_2$};
\draw [decorate,decoration={brace,amplitude=5pt}] (1.45,1.1) -- (1.45,2.4) node[midway,xshift=-20pt]{\small$f_1\!=\!e_1$};
\draw [decorate,decoration={brace,amplitude=5pt,mirror}] (1.55,2.6) -- (1.55,3.4) node[midway,xshift=10pt]{\small$f_2$};

\draw (3.25,1) node[above]{\small$p_0$};
\draw[fill=black] (3,1) circle[radius=2pt];

\draw (1.75,3.5) node[above]{\small$p_1$};
\draw[fill=black] (1.5,3.5) circle[radius=2pt];

\draw (4.25,2.5) node[above]{\small$p_2$};
\draw[fill=black] (4.5,2.5) circle[radius=2pt];

\draw (5,-1) -- (0,4);

\end{tikzpicture}
\]
\caption{Scaffold}
\label{fig: biperm scaffold}
\end{subfigure}\qquad\qquad
\begin{subfigure}[b]{0.3\textwidth}
\[
\begin{tikzpicture}[scale=1.2]

\draw (0,0) -- (1,0) -- (1,1) -- (0,1) -- (0,0);

\draw (0,1) -- (1,1) -- (1.75,1.75) -- (1.75,2.75) -- (0.75,2.75) -- (0,2) -- (0,1);

\draw (1,0) -- (2,0) -- (2.75,0.75) -- (2.75,1.75) -- (1.75,1.75) -- (1,1) -- (1,0);

\draw (2,0) -- (3,0) -- (3.75,0.75) -- (2.75,0.75) -- (2,0);

\draw (2.75,0.75) -- (3.75,0.75) -- (3.75,1.75) -- (2.75,1.75) -- (2.75,0.75);

\draw (2.75,1.75) -- (3.75,1.75) -- (3.75,2.75) -- (2.75,2.75) -- (2.75,1.75);

\draw (1.75,1.75) -- (2.75,1.75) -- (2.75,2.75) -- (1.75,2.75) -- (1.75,1.75);

\draw (2.75,2.75) -- (3.75,2.75) -- (3.75,3.75) -- (2.75,3.75) -- (2.75,2.75);

\draw (1.75,2.75) -- (2.75,2.75) -- (2.75,3.75) -- (1.75,3.75) -- (1.75,2.75);

\draw (0.75,2.75) -- (1.75,2.75) -- (1.75,3.75) -- (0.75,3.75) -- (0.75,2.75);

\draw (0,2) -- (0,3) -- (0.75,3.75) -- (0.75,2.75) -- (0,2);

\draw[fill=black] (1.875,0.875) circle[radius=1.75pt];
\draw (1.925,0.875) node[below]{\small$x_0$};

\draw[fill=black] (1.2,3.35) circle[radius=1.75pt];
\draw (1.25,3.35) node[below]{\small$x_1$};

\draw[fill=black] (3.3,2.4) circle[radius=1.75pt];
\draw (3.35,2.4) node[below]{\small$x_2$};

\end{tikzpicture}
\]
\caption{Expansion}
\label{fig: biperm expansion}
\end{subfigure}
\caption{A stratum of $P_{\Lambda_B}$.}
\end{figure}
\noindent The associated bisequence is $2|0|12|1$ and the corresponding cone $\uptau \in \Sigma_{n,n}$ is the subset of $\uprho$ defined by the additional (in)equalities
\begin{align*}
a_0 + b_0 & \leq a_1 + b_1, \\
a_0 + b_0 & = a_1 + b_2, \\
a_0 + b_0 & \leq a_2 + b_1, \\
a_0 + b_0 & \leq a_2+b_2.
\end{align*}
The last two inequalities are redundant: $a_0 \leq a_2$ and $b_0 \leq b_2 \leq b_1$ on $\uprho$ already imply $a_0 + b_0 \leq a_2 + b_1$ and $a_0 + b_0 \leq a_2 + b_2$. The first two (in)equalities are essential.

The polyhedral edge lengths again give a coordinate system on $\uptau$. Notice that now we have $e_1=f_1$ because the antidiagonal is required to pass through the vertex $(a_0-e_1,b_0+f_1)$. This arises from the equality
\[ a_0+b_0 = a_1 + b_2 \Leftrightarrow a_0-a_1 = b_2-b_0 \Leftrightarrow e_1=f_1.\]
A coordinate system for $\uptau$ is therefore given by $(e_1,e_2,f_2)$. Correspondingly we have $\dim \uptau = 3$ whereas $\dim \uprho = 4$.

The associated tropical expansion is illustrated in Figure~\ref{fig: biperm expansion}. The $2$ hexagonal components are two-dimensional permutahedral varieties and each of the remaining $9$ components is a $\PP^1 \times \PP^1$. As far as the moduli of the $x_i$ is concerned, this expansion is for all intents and purposes identical to the expansion considered in Example~\ref{example: square of permutahedron}. The rubber torus, however, differs in a crucial way: it is now only three-dimensional
\[ T_\uptau = (\Gm^3)_{e_1 e_2 f_2}. \]
This is precisely the subtorus $T_\uptau \subseteq T_\uprho$ that acts trivially on the divisor joining the two hexagonal components; see \cite[Example~4.4]{CarocciNabijouRubber} for a similar phenomenon. As usual the rubber action is governed by the tropical position maps
\begin{align*}
\varphi_{v_0}(e_1,e_2,f_2) & = (a_0,b_0) = (0,0), \\
\varphi_{v_1}(e_1,e_2,f_2) & = (a_1,b_1) = (-e_1,e_1+f_2), \\
\varphi_{v_2}(e_1,e_2,f_2) & = (a_2,b_2) = (e_2,e_1).	
\end{align*}
From Theorem~\ref{thm: strata} we see that there are $(2+2)-3=1$ dimensions of moduli for $P_{\Lambda_B,\uptau} \subseteq P_{\Lambda_B}$. This is consistent with the fact that $\operatorname{codim}\uptau=1$.
\end{example}

\subsection{Harmonic fan} \label{sec: harmonic} A close cousin of the bipermutahedral fan is the harmonic fan $H_{n,n}$, introduced in \cite[Section~2.8]{ArdilaDenhamHuh} and studied in \cite{ArdilaEscobar}. It is non-simplicial and sits in a tower of refinements
\[ \Sigma_{n,n} \to H_{n,n} \to \Sigma_n \times \Sigma_n.\]
A cone of $H_{n,n}$ is produced by choosing pairwise comparisons of $a_0,\ldots,a_n$ and of $b_0,\ldots,b_n$, along with a subset of $\{0,\ldots,n\}$ indexing those points $p_i$ which lie on the supporting antidiagonal.

\begin{question} Does $H_{n,n}$ arise as the configuration fan $\Pi_\Lambda$ associated to any tropical scaffold $\Lambda$? \end{question}

We are unable to identify a suitable scaffold. If the answer is indeed negative, this provides another example of the bipermutahedral fan enjoying properties which the harmonic fan lacks.



\footnotesize
\bibliographystyle{alpha}
\bibliography{Bibliography.bib}

\smallskip

\noindent Navid Nabijou. Queen Mary University of London. \href{mailto:n.nabijou@qmul.ac.uk}{n.nabijou@qmul.ac.uk}

\end{document}